\documentclass[12pt]{amsart}
\usepackage{amsfonts}
\usepackage{amsfonts,latexsym,rawfonts,amsmath,amssymb,amsthm}
\usepackage{amsmath,amssymb,amsthm,amscd,esint}
\usepackage[plainpages=false]{hyperref}
\usepackage{appendix}

\usepackage{graphicx}
\usepackage{float}
\usepackage{subfigure}

\RequirePackage{color}

\numberwithin{equation}{section}

\newcommand{\beq}{\begin{equation}}
\newcommand{\eeq}{\end{equation}}
\newcommand{\beqs}{\begin{eqnarray*}}
\newcommand{\eeqs}{\end{eqnarray*}}
\newcommand{\beqn}{\begin{eqnarray}}
\newcommand{\eeqn}{\end{eqnarray}}
\newcommand{\beqa}{\begin{array}}
\newcommand{\eeqa}{\end{array}}

\newcommand{\cL}{{\mathcal L}}

\newcommand{\cU}{{\mathcal U}}

\newcommand{\sq}{{\sqrt{-1}}}

\newcommand{\psho}{\mathcal{PSH}_0}

\newcommand{\ddc}{\sq\p\bar \p}
\newcommand{\lp}{\left(}
\newcommand{\rp}{\right)}
\newcommand{\lf}{\left[}
\newcommand{\rf}{\right]}
\newcommand{\ls}{\left\{}
\newcommand{\rs}{\right\}}

\newcommand{\lan}{\langle}
\newcommand{\ran}{\rangle}
\newcommand{\D}{\nabla}
\newcommand{\p}{\partial}
\newcommand{\eps}{\varepsilon}
\newcommand{\lam}{\lambda}
\newcommand{\Om}{\Omega}

\newtheorem{prop}{Proposition}[section]
\newtheorem{theo}[prop]{Theorem}
\newtheorem{lem}[prop]{Lemma}

\newtheorem{cor}[prop]{Corollary}
\newtheorem{rem}[prop]{Remark}

\allowdisplaybreaks
\begin{document}
\baselineskip=16.4pt
\parskip=3pt

\title{\large S\MakeLowercase{obolev} \MakeLowercase{inequalities and regularity of the linearized 
complex} M\MakeLowercase{onge}-A\MakeLowercase{mp\`ere}  \MakeLowercase{and} H\MakeLowercase{essian equations}}
\author{J\MakeLowercase{iaxiang} W\MakeLowercase{ang} \MakeLowercase{and} B\MakeLowercase{in} $\text{Z\MakeLowercase{hou}}^{1}$}

\begin{abstract} Let $u$ be a smooth, strictly $k$-plurisubharmonic function on a bounded domain $\Omega\in\mathbb C^n$ with $2\leq k\leq n$. 
The purpose of this paper is to study the regularity of solution to the linearized complex Monge-Amp\`ere and Hessian equations when the complex $k$-Hessian $H_k[u]$ of $u$ is bounded from above and below. We first  establish some estimates of Green's functions associated to the linearized equations. Then we prove a class of new Sobolev inequalities. 
With these inequalities, we use Moser's iteration to investigate the a priori estimates of Hessian equations and their linearized equations, as well as the K\"ahler scalar curvature equation. In particular, we obtain the Harnack inequality for the linearized complex Monge-Amp\`ere and Hessian equations under an extra integrability condition on the coefficients. 
The approach works in both real and complex case.

\end{abstract}

\address{Jiaxiang Wang: 
School of Mathematical Sciences and LPMC, Nankai University, Tianjin, 300071, P. R. China}
\email{wangjx\underline{ }manifold@126.com}

\address{Bin Zhou:
School of Mathematical Sciences, Peking
University, Beijing 100871, China.} 
\email{bzhou@pku.edu.cn}

\thanks {$1$ Partially supported by National Key R$\&$D Program of China SQ2020YFA0712800 and NSFC  Grant 11822101.}

\subjclass[2000]{Primary: 32W20; Secondary: 35J60.}

\keywords{Complex Monge-Amp\`ere equation; Hessian equation; linearized Monge-Amp\`ere equation; Sobolev inequality.}

\maketitle

\section{Introduction}

In a seminar work \cite{CG}, Caffarelli-Gutierrez established a fundamental regularity theory of the
linearized Monge-Amp\`ere equation
\beq\label{eq}
L_uv:=U^{ij}v_{i j}=f(x),
\eeq
with coefficient matrix $(U^{ij})$  given as the cofactor
matrix of  the Hessian $D^2u$ of a smooth, strictly convex function  $u$, which is analogous to the De Giorgi-Nash-Moser theory of uniformly elliptic divergence form linear equations 
and the Krylov-Safonov theory of general form linear equations. Precisely, they obtained the Harnack and H\"older estimates
under the assumption that  the Monge-Amp\`ere measure $\mu_u$ of $u$ satisfies the $(\mu_\infty)$ condition \cite{CG}; see \cite{M1} for the extensions.  In particular,
the $(\mu_\infty)$ condition is satisfied when there exists two constants $\lambda, \Lambda>0$, such that
\beq\label{det-c}
0<\lambda\leq {\det} (u_{i j})\leq \Lambda.
\eeq
Caffarelli-Gutierrez's theory plays an important role in many geometric problems \cite{TW1, TW2, D1, D2, D3}.
In this paper, we study the regularity for the linearized equations for a class of equations, including the complex Monge-Amp\`ere equation and Hessian equations in both real and complex cases.

Let  $\Omega\in \mathbb C^n$ be a bounded domain and $u$ a smooth function defined on $\Omega$.
The $k$-complex Hessian operator is defined by 
\beq\label{k-c-hess}
H_k[u]=\sum_{1\leq j_1<\cdots<\lambda_k\leq n}\lambda_{j_1}\cdots\lambda_{j_k}, \ k=1, \cdots, n,
\eeq
where $\lambda_1,\cdots,\lambda_n$ are the eigenvalues of the complex Hessian $(u_{i\bar j})=\left(\frac{\partial^2u}{\partial z_i\partial\bar z_j}\right)$.
Using $d=\partial+\bar\partial$ and $d^c=\sqrt{-1}(\bar\partial-\partial)$, one gets
\[(dd^cu)^k\wedge \omega_0^{n-k}=4^nk!(n-k)!H_k[u]\,d\mu_0,\]
where $\omega_0=dd^c(|z|^2)$ is the fundamental K\"ahler form and $d\mu_0$ is the volume form.
$u$ is called {\it(strictly) $k$-plurisubharmonic}  if $H_l[u]\geq 0(>0)$, $l=1, 2, \cdots, k$.

Denote $H^{i\bar j}_k=\frac{\partial H_k[u]}{\partial u_{i\bar j}}$. For a $k$-plurisubharmonic function, the {\it linearized complex $k$-Hessian equation}
is 
\beq\label{l-k-c-hess}
\mathcal L_{u,k}v:=H^{i\bar j}_k v_{i\bar j}=f(z).
\eeq
In particular, when $k=n$, $H_n$ is the complex Monge-Amp\`ere operator and $\mathcal L_{u,n}$ is the linearized complex Monge-Amp\`ere 
equation
\beq\label{l-cma}
\mathcal L_{u,n}v:=U^{i\bar j} v_{i\bar j}=f(z),
\eeq
where $(U^{i\bar j})$ is the cofactor matrix of $(u_{i\bar j})$, or equivalently,
\beq\label{l-cma1}
\det(u_{i\bar j})\triangle_\omega v=f(z),
\eeq
where $\triangle_\omega$ is the Laplacian with respect to the K\"ahler metric $\omega=dd^cu$. This equation 
has many  potential applications in the regularity of complex Monge-Amp\'ere equation and  K\"ahler geometry. 

Caffarelli-Gutierrez's theory for the linearized real Monge-Amp\`ere equation relies on the convexity and real analysis related to the Monge-Amp\`ere equation, 
which do not hold for the Hessian equations or in the complex setting. 
Recently, an interior H\"older estimate for the linearized complex Monge-Amp\`ere equation was obtained by \cite{X}
under the assumption that the domain $\Omega$ is close to a ball, $u$ vanishes on $\partial\Omega$ and the determinant $\det u_{i\bar j}$ is close to a constant.
Except this result, very little is known for the linearized complex Monge-Amp\`ere and Hessian
equations. 

By the divergence free property of $H_k$, i.e., $\sum_{j=1}^n D_{\bar j}H_k^{i\bar j}=0$, we can rewrite \eqref{l-k-c-hess} in divergence form
\beq\label{com-div-1}
D_{\bar j}(H_k^{i\bar j}D_i v)=f(z).
\eeq
Hence, it is natural to investigate \eqref{l-k-c-hess} by integral methods. 


The key tool used later in the study of regularity is a class of new Sobolev inequalities associated to the linearized operators which is of interest itself. For the real Monge-Amp\`ere equation, such an inequality was obtained by \cite{TW} when $n\geq 3$ and  \cite{L2} when $n=2$. In fact,
when the Monge-Amp\`ere measure $\mu_u$ satisfies the $(\mu_\infty)$ condition in \cite{CG} and the sections of $u$ satisfy certain size conditions, it was shown that the following
inequality holds
\begin{align}\label{r-so}
\|v\|_{L^p(\Omega,\mu_u)}\le C\cdot \lp \int_{\Om}U^{i j}v_i v_{j}\,dx\rp^{\frac{1}{2}}, \ v\in C^{1}_0(\Omega),
\end{align}
where $p\geq 1$ depends on $n$, $\Om$ and $u$. In particular, \eqref{r-so} holds for $p=\frac{2n}{n-2}$ when \eqref{det-c} is satisfied.
In \cite{M2}, it was proved that
\eqref{r-so} holds for $p=\frac{2n}{n-1}$ when  $\mu_u$ satisfies the doubling property.
One of the main ingredients of the proof is 
a power-like decay estimate of the distribution function of the Green's function of the linearized Monge-Amp\`ere equation. 
The proofs in these references rely  on the Harnack inequality of solutions to the linearized Monge-Amp\`ere equation.
However, in the Hessian case or the complex setting, 
the regularity of the linearized equations is what we are seeking to.
New techniques are needed to investigate this problem. We will use Chen-Cheng's method of comparison with auxiliary complex Monge-Amp\`ere equation in the study of 
constant scalar K\"ahler metrics \cite{CC}
to estimate the Green's function.  

Properties of Green's functions for uniformly elliptic operators have been studied extensively \cite{GW, LSW}. Green's function for the linearized real Monge-Amp\`ere equation was studied by in \cite{TW, L1}. In the complex case, an $L^{\frac{n}{n-1}-\epsilon}$-estimate of Green's functions on K\"ahler manifolds were obtained in \cite{GPS}. 
We are concerned with the local case and the critical exponent.
By \cite{GW, LSW}, there exists a Green's function $G_\Omega$ associated to $\mathcal L_{u,k}$ in $\Om$, which satisfies the following Dirichlet problem for any fixed $w\in \Om$:
\begin{align}\label{Green}
\begin{cases}
-H_k^{i\bar j}\lf G_{\Om}(z,w)\rf_{i\bar j}=\delta_w,\ \ \ \ & z\in \Om,  \\[4pt]
G_{\Om}(z,w)=0,   & z\in \p \Om.
\end{cases}
\end{align}
Our first result is 
\begin{theo}[Estimates of Green's functions]\label{main1}
Assume $u$ is a smooth, strictly $k$-plurisubharmonic function$(2\leq k\leq n)$ and $\mu$ is the associated $k$-Hessian measure.
\begin{itemize}
\item[(i)] Suppose $H_k[u]\leq \Lambda$ for $\Lambda>0$. There exists $C>0$ depending only on $k$, $n$, $\Lambda$ and $\Omega$ such that for any $z\in\Omega$, 
\begin{align}\label{w-lp}
\|G(z,\cdot)\|_{L^{\frac{n}{n-1}}_*(\Omega,\,d\mu)}:=\sup_{t>0}\{t\cdot\mu_u( \ls w\in \Omega\,\big|\,G_{\Omega}(z,w)>t\rs )^{\frac{n-1}{n}}\}\le C.
\end{align}

\item[(ii)]  Let $\Lambda_q=\|H_k[u]\|_{L^{q}(\Omega)}$ for some $q>1$. Then for any $p<\frac{n}{n-1}$,  there exists a uniform constant $C>0$ depending only on $k$, $n$, $p$, $\Omega$ and $\Lambda_q$ such that for any $z\in\Omega$,
\begin{align}\label{lp-green}
\|G(z,\cdot)\|_{L^{p}(\Omega,\,d\mu)}:=\left(\int_{\Om}G_{\Om}(z,\cdot)^pH_k[u]\,\omega_0^n\right)^{\frac{1}{p}}\le C. 
\end{align}
\end{itemize}
\end{theo}

In the following, we will always  assume 
$u$ satisfies
\begin{equation}\label{low-upp}
0<\lambda\leq H_k[u]\leq \Lambda.
\end{equation}  We obtain the following new Sobolev inequality associated to $\mathcal L_{u,k}$.



\begin{theo}[Sobolev inequality]\label{Sobolev}
Assume $u$ is a smooth, strictly $k$-plurisubharmonic function$(2\leq k\leq n)$satisfying \eqref{low-upp}. Then for any 
$2\le p\le  \frac{2n}{n-1}$. There exists a constant $C$ depending on $k$, $n$, $\lambda$, $\Lambda$, $\Omega$ and $p$ such that for any $v\in C^{2}_0(\Om)$, we have
\beq\label{sob-ineq}
\|v\|_{L^p(\Omega,\mu_u)}\le C\cdot \lp \int_{\Om}H_k^{i\bar j}v_i\bar v_{j}\,\omega_0^n\rp^{\frac{1}{2}}=C\cdot\lp \int_{\Om}\partial v\wedge\bar\partial v\wedge (dd^cu)^{k-1}\wedge \omega_0^{n-k}\rp^{\frac{1}{2}}.
\eeq
\end{theo}

Since our proof does not need the regularity of the linearized equations, it also provides a new proof to \eqref{r-so} in the real case in \cite{TW} without using \cite{CG}.  
It can be used to derive the analogous Sobolev inequalities associated to linearized real  $k$-Hessian equations(see Section \ref{sec-hes}). 

A first application of the new Sobolev inequality is the local regularity of the complex Monge-Amp\`ere and $k$-Hessian equations.
The local estimates of  $W^{2,p}$ solutions to  
\beq\label{k-hess-c}
H_k[u]=f
\eeq 
were obtained in \cite{BD, DK}. The proofs in \cite{BD, DK} use Alexandrov-Bakeman-Pucci's maximum principle. By using the Sobolev inequalities above, we give a new proof under weaker regularity of $f$.
\begin{theo}[Interior estimate of Hessian equations]\label{main2}
Assume $n\geq 2$ and $p> n(k-1)$. Suppose $\Omega$ is a domain in $\mathbb C^n$. 
Suppose $u\in W^{2,p}(\Omega)$ is a solution to \eqref{k-hess-c} with $0<\lambda\leq f\leq \Lambda$. Then for $\Omega'\Subset\Omega$, we have 
\begin{align}
\sup_{\Omega'}|\triangle u|\le C,
\end{align}
where $C$ depends on 
$k$, $n$, $p$, $\lambda$, $\Lambda$, $\|f\|_{C^1(\Omega)}$, $\|u\|_{W^{2,p}(\Omega)}$ and $dist(\Omega', \p\Omega)$.
\end{theo}

Now we turn to the regularity of \eqref{l-k-c-hess}.
Unlike the linearized real Monge-Amp\`ere equation, we need some extra conditions on the coefficients $H_k^{i\bar j}$.
For convenience, we denote $\mathcal U=\sum_{i=1}^nH_k^{i\bar i}$. For $q>0$, $R>0$, and $B_R(z)\subset \Omega$, we denote 
$$\Theta_q(B_R(z))=\lp\frac{1}{|B_R(z)|}\int_{B_R(z)}\mathcal U^q\,d\mu_0\rp^{\frac{1}{q}}.$$

\begin{theo}[Harnack inequality]\label{thm-har3}
Assume $u$ is a smooth, strictly $k$-plurisubharmonic function($2\leq k\leq n$) that satisfies 
\eqref{low-upp}.
Assume $ f(z)\in L^q(B_{R_0}(z_0))$ with $q>n$.
Let $v\in W^{1,2}(B_{R_0}(z_0))$ be a nonnegative solution to \eqref{com-div-1} in $B_{R_0}(z_0)\subset\Omega$. Then there holds for any $n<q_0$ and $R<R_0$
\beq\label{cma-har}
\max_{B_R(z_0)} v\leq C\frac{\Theta_{q_0}^{\frac{\delta}{4}}(B_{R_0}(z_0))}{(R_0-R)^n}\cdot\lp\min_{B_{R_0}(z_0)} v+{R_0}^{2-\frac{2n}{q}}\|f\|_{L^{q}(B_{R_0}(z_0))}\rp.
\eeq
where $\delta=\max\{\frac{q_0n}{q_0-n}, \frac{qn}{q-n}\}$ and $C$ depends only on $k$, $n$, $\lambda$, $\Lambda$, $\|f\|_{L^q(B_{R_0}(z_0))}$.
\end{theo}

For the uniformly elliptic equations, we can usually establish the H\"older continuity of solutions from the Harnack inequality. Unfortunately, since the constant  in \eqref{cma-har} depends on $\Theta_{q_0}(B_{R_0}(z_0))$, we need a stronger assumption to get the H\"older continuity.

\begin{theo}[Interior H\"older estimate]\label{int-holder}
Assume $u$ is a smooth, strictly $k$-plurisubharmonic function($2\leq k\leq n$) that satisfies 
\eqref{low-upp}. Assume $f(z)\in L^q(\Omega)$ with $q>n$. Let $v\in W^{1,2}(\Omega)$ be a solution to \eqref{com-div-1}. Suppose $\Theta_{q_0}(B_R(z))$ is uniformly bounded for $B_R(z)\subset\Omega$ for some $q_0>n$.
Then $v\in C^\alpha(\Omega)$ for some 
$\alpha\in(0,1)$ depending on $k$, $n$, $q_0$, $q$, $\lambda$, $\Lambda$, $\|f\|_{L^q(\Omega)}$ and $\Theta_{q_0}=\sup_{B_R(z)\subset\Omega}{\Theta_{q_0}(B_R(z))}$. Moreover, there holds for any $B_R\subset\Omega$
\[|v(z)-v(w)|\leq C\lp\frac{|z-w|}{R}^\alpha\rp(\|v\|_{L^2(B_R)} +R^{2-\frac{2n}{q}}\|f\|_{L^{q}(B_R)})\]
for $z, w\in B_{R/2}$ where the constant $C$ depends on $k$, $n$, $q_0$, $q$, $\lambda$, $\Lambda$, $\|f\|_{L^q(\Omega)}$ and $\Theta$. 
\end{theo}

\begin{rem}
 For the general non-uniformly elliptic equation of divergence form
\beq\label{de-eq}
D_j(a_{ij}(x)D_i u)=D_ig^i + h\quad\text{in}\,\,\Omega\subset\mathbb R^n,
\eeq
where $g=(g^1, \cdots, g^n)$ is a vector valued function and $\{a_{ij}(x)\}$ satisfy
$$0< \lambda(x)|\xi|^2\leq a_{ij}(x)\xi_i\xi_j\leq \Lambda(x)|\xi|^2,\ \ \forall\ \xi\in\mathbb R^n\setminus\{0\},$$
 the classical De Giorgi-Nash-Moser's theory was extended in \cite{MS,T1, T2}
with integrability conditions on $\lambda(x)^{-1}$ and $\Lambda(x)$. 
In particular, the local boundedness and Harnack inequality were obtained by Trudinger \cite{T1, T2} when  $\lambda^{-1}\in L_{loc}^p(\Omega)$ for some $p>n$, $\Lambda\in L^q(\Omega)$, $g\in L^{\infty}(\Omega;\mathbb{R}^n)$ and $h\in L^q(\Omega)$  satisfying 
$\frac{1}{q}+\frac{1}{p}<\frac{2}{n}$.
Very recently, it is proved that  the conclusion holds under the optimal assumption
\beq\label{pq-cond1}
\frac{1}{q}+\frac{1}{p}<\frac{2}{n-1}.
\eeq 
since it is also known to be false when $\frac{1}{q}+\frac{1}{p}<\frac{2}{n-1}+\epsilon$ \cite{FSS}.
Our results do not automatically follow from Bella-Sch\"affner's results. Indeed, let $\hat\lambda$ be and $\hat\Lambda$ be the smallest and largest eigenvalues of the coefficient matrix of \eqref{com-div-1}, respectively.  Then $\hat\Lambda\in L^{q_0}$ by \eqref{low-upp}, we have $\hat\lambda^{-1}\leq C\hat\Lambda^{n-1}\in L^{\frac{q_0}{n-1}}$. Since the real dimension is $2n$ and $q_0>n$, condition \eqref{pq-cond1}  is not satisfied.
\end{rem}

The structure of the paper is as follows. In Section \ref{sec-gr}, we  derive the integral estimates for Green's functions of the linearized complex $k$-Hessian equations. Here we used  Chen-Cheng's method of comparison with auxiliary complex Monge-Amp\`ere equation. In particular, to get the critical $L^{\frac{n}{n-1}}_*$ estimate, we need the Moser-Trudinger inequality for the
complex Monge-Amp\`ere equation \cite{BB, WWZ1}.
The new Sobolev inequalities associated to the linearized complex Monge-Amp\`ere and Hessian operators  are established in Section \ref{sec-sob}.  The  $L^{\frac{n}{n-1}}_*$ estimate is needed to obtain the critical exponent in \eqref{sob-ineq}. 
In Section \ref{sec:local}, we present the proof of Theorem \ref{main2}.
Then  following lines in \cite{HL, T2}, we use the new Sobolev inequality to prove Theorem \ref{thm-har3} in Section \ref{sec-har}.
In particular, we prove the local boundedness(Theorem \ref{thm-har}) and weak Harnack inequality(Theorem \ref{weak-har2}).
The proofs use classic Moser's iteration. A new Poincar\'e type inequality(Lemma \ref{p-ineq}) is also needed in the proof of Theorem \ref{weak-har2}. However, in our case, we only have $(\frac{2n}{2n-1},2)$-type Poincare inequality,  instead of  the  $(2,2)$-type Poincare inequality in usual case. We will do more iterations to settle this problem. Section \ref{sec-hes} is devoted to some remarks on the linearized real $k$-Hessian equations. In Section \ref{sec-scalar}, we apply the results in Section \ref{sec-har} to obtain some interior estimates and a Liouville theorem for the scalar curvature equation.  

{\bf Acknowledgments.} The authors would like to thank Guji Tian and Ling Wang for some discussions. We also thank Nam Le for bringing \cite{BS} to our attention.

\vskip 20pt

\section{Integrability of Green's function}\label{sec-gr}

In this section, we derive some integral estimates for Green's functions associated to $\mathcal L_{u,k}$. The idea is to reduce the desired estimates to the $L^\infty$-estimate of auxiliary complex Monge-Amp\`ere equations. To obtain the critical exponent, we need the following Moser-Trudinger inequality for the complex Monge-Amp\`ere equation. 

\begin{theo}\cite{BB, WWZ1} \label{MT}
	Let $\Omega$ be a bounded, smooth, pseudo-convex domain. Let $\mathcal{PSH}_0(\Omega)$  be the set of pluri-subharmonic functions 
which vanish on $\partial \Omega$. 
	Then $\forall\ u\in \mathcal{PSH}_0(\Omega)\cap C^{\infty}(\bar \Omega)$, $ u\not\equiv 0$
	\beq\label{MT0}
	\int_\Omega e^{ \alpha \left(\frac{- u}{\|u\|_{\mathcal{PSH}_0(\Omega)}}\right)^{\frac{n+1}{n}}} \leq C . 
	\eeq
	where $\alpha$ and $C$ depends on $n$ and $\text{diam}(\Omega)$. Here
	\beq
\|u\|_{\mathcal{PSH}_0(\Omega)}=\left[\frac{1}{n+1}\int_{\Omega}(-u)(dd^cu)^n\right]^{\frac{1}{n+1}}.
\eeq
\end{theo}

\begin{proof}[Proof of Theorem \ref{main1}]
The proof is inspired by Chen-Cheng's approach of comparison with auxiliary complex Monge-Amp\`ere equations in the study of constant scalar K\"ahler metrics \cite{CC}, which makes uses of the algebraic structure of Monge-Amp\`ere and Hessian operators.

First, we show the weak $L^{p}_*$ estimate of Green's function for  $p=\frac{n}{n-1}$ if $0\le H_k[u]\le \Lambda$. The proof of case $(ii)$ is similar, which needs only a few modifications. 
For each $z\in\Omega$ and $t>0$, let $K_t:=\ls G_{\Omega}(z,w)>t\rs$. We are going to estimate $t^{\frac{n}{n-1}}\mu_u(K_t)$. 

Let $v^t$ be the solution to the following Dirichlet problem
\begin{align}
\begin{cases}
	\cL_{u,k}v^t=H_k^{i\bar j}\lp v^t\rp_{i\bar j}=t^{\frac{1}{n-1}}\chi_{K_t}H_k[u],\ \ \ &\text{in }\Omega,  \\[4pt]
	\ \ \ \ \ v^t=0,   &\text{on }\p\Omega.
\end{cases}
\end{align}
For any $z\in \Omega$, we have
$$|v^t(z)|=\left|\int_{K_t}G_{\Omega}(z,\cdot)t^{\frac{1}{n-1}}\,d\mu_u\right|\ge t^{\frac{n}{n-1}}\mu_u(K_t).$$
Let $\psi$ be a smooth plurisubharmonic function which solves the following equation
\begin{align}
\begin{cases}
	H_n[\psi]=\det(\psi_{i\bar j})=t^{\frac{n}{n-1}}\chi_{K_t}H_k^{\frac{n}{k}}[u],\ \ \ &\text{in }\Omega,  \\[4pt]
	\ \ \ \ \ \psi=0,   &\text{on }\p\Omega.
\end{cases}
\end{align}
We have 
\begin{eqnarray*}
H_k^{i\bar j}\psi_{i\bar j}\ge kH_k[u]^{1-\frac{1}{k}}H_k[\psi]^{\frac{1}{k}}&\ge& k\binom{n}{k}^{\frac{1}{k}}H_k[u]^{1-\frac{1}{k}}H_n[\psi]^{\frac{1}{n}}\\
&=& C_{n,k}H_k[u] t^{\frac{1}{n-1}}\chi_{K_t}=\cL_{u,k}\lp C_{n,k} v^t\rp.
\end{eqnarray*}
Therefore,
\begin{align}\label{infty-2}
t^{\frac{n}{n-1}}\mu_u(K_t)\le \|v^t\|_{L^{\infty}(\Omega)}\le \frac{1}{C_{n,k}}\|\psi\|_{L^{\infty}(\Omega)}.
\end{align}
Now we set $\phi:=t^{-\frac{1}{n-1}}\psi$, which satisfies $H_n[\phi]=\chi_{K_t}H_k[u]^{\frac{n}{k}}$. 
Hence, it suffices to obtain the $L^\infty$-estimate for $\phi$. We use
 Theorem \eqref{MT} and the iteration in \cite{WWZ2}. 

Denote $\hat \mu(E):=\int_E(\ddc\phi)^n=\mu_u(K_t\cap E)$, and $\Omega_s:=\ls \phi<-s\rs$. For $r>0$, we have
\begin{align}
r\hat\mu\lp \Omega_{r+s}\rp\le & \int_{\Omega_s}(-\phi-s)(\ddc\phi)^n   \nonumber\\
\le & \left[ \int_{\Omega_s}\lp \frac{-\phi-s}{\|\phi+s\|_{\psho(\Omega_s)}}\rp^{\beta}\,\chi_{K_t}d\mu_u\right]^{\frac{1}{\beta}}\cdot \hat\mu\lp \Omega_s\rp^{1-\frac{1}{\beta}}\cdot \|\phi+s\|_{\psho(\Omega_s)}, \label{infty-1}
\end{align}
where $\beta>0$ and
\[\|\phi+s\|_{\psho(\Omega_s)}:= \left[ \int_{\Omega_s}(-\phi-s)\lp \ddc \phi\rp^n\right]^{\frac{1}{n+1}}.\]
To estimate the right-hand side of \eqref{infty-1}, we recall the following version of Young's inequality: 
for $x$, $y>0$, let $\eta(t)=t \log^{\frac{n\beta}{n+1}}(1+t)$, then there holds
\begin{align}
x\cdot y \le &   \int_0^x\eta(t)dt+\int_0^{y}\eta^{-1}(s)ds  \nonumber\\
\le & x \log^{\frac{n\beta}{n+1}}(1+x)+\int_0^{\exp(y^{\frac{n+1}{\beta n}})}\frac{s}{1+s}\log^{\frac{n\beta}{n+1}-1}(1+s)\,ds\nonumber\\[5pt]
\leq &  x \log^{\frac{n\beta}{n+1}}(1+x)+\exp( 2y^{\frac{n+1}{\beta n}}).  \label{Young}
\end{align}
Let $x=1$, $y=\lp\frac{\alpha}{2}\rp^{\frac{n\beta}{n+1}}\cdot\lp \frac{-\phi-s}{\|\phi+s\|_{\psho(\Omega_s)}}\rp^{\beta}$ in \eqref{Young}, where $\alpha>0$ is the constant in \eqref{MT0}. By $0\le H_k[u]\le \Lambda$ we can get
\begin{align}\label{case2}
& \left[ \int_{\Omega_s}\lp \frac{-\phi-s}{\|\phi+s\|_{\psho(\Omega_s)}}\rp^{\beta}\,\chi_{K_t}d\mu_u\right]^{\frac{1}{\beta}}\nonumber\\
\le & \Lambda^{\frac{1}{\beta}}\lp\frac{2}{\alpha}\rp^{\frac{n}{n+1}}\lf \int_{\Omega_s}\lp\frac{\alpha}{2}\rp^{\frac{n\beta}{n+1}}\cdot\lp \frac{-\phi-s}{\|\phi+s\|_{\psho(\Omega_s)}}\rp^{\beta}\,d\mu_0 \rf^{\frac{1}{\beta}}  \nonumber\\[5pt]
\le & C(n,\Lambda,\Omega).
\end{align}
Note that the constants in the Moser-Trudinger inequality depend only on the diameter of the domain. 
Therefore, by \eqref{infty-1},  we obtain
\[ \int_{\Omega_s}(-\phi-s)(\ddc\phi)^n=\|\phi+s\|_{\psho(\Omega_s)}^{n+1}\leq \hat\mu\lp  \Omega_s\rp^{1-\frac{1}{\beta}}\cdot \|\phi+s\|_{\psho(\Omega_s)},\]
from which we deduce
\begin{align}
r\hat\mu(\Omega_{s+r})\le C\hat\mu(\Omega_s)^{1+\frac{\beta-n-1}{n\beta}},
\end{align}
where $C>0$ depends only on $k$, $n$, $\Omega$, $\Lambda$.
Letting $\beta\to \infty$, we have
\begin{align}\label{ite}
r\hat\mu(\Omega_{s+r})\le C\hat\mu(\Omega_s)^{1+\frac{1}{n}}.
\end{align}
In particular,
$$\hat\mu(\Omega_s)\le \frac{C}{s}\mu_u (K_t)^{1+\frac{1}{n}}$$
for some $C>1$.
Then we choose $s_0=2^{n+1}C^{n+1}\mu_u (K_t)^{\frac{1}{n}}$ such that $\hat\mu (\Omega_{s_0})\leq \frac{1}{2}\mu_u (K_t)$. For any $l\in \mathbb{Z}_+$, define 
\begin{align}\label{iteration}
	s_l=s_0+\sum_{j=1}^l2^{-\frac{j}{n}  }\mu_u (K_t)^{\frac{1}{n}},\ \ \phi^l=\phi+s^l, \ \ \Omega_l=\Omega_{s_l}.
\end{align}
Then
$$2^{-\frac{l+1}{n}}\mu_u (K_t)^{\frac{1}{n} }\hat\mu (\Omega_{l+1})=(s_{l+1}-s_l)\hat\mu (\Omega_{l+1})\le C\hat\mu (\Omega_l)^{1+\frac{1}{n}}.$$
We claim that $|\Omega_{l+1}|\leq \frac{1}{2}|\Omega_l|$ for any $l$. By induction, we assume the inequality holds for $l\leq m$. 
Then
\begin{align*}
	\hat\mu (\Omega_{m+1})\leq &  C\hat\mu (\Omega_m)^{1+\frac{1}{n} }\frac{2^{\frac{m+1}{n}}}{\mu_u (K_t)^{\frac{1}{n} }}    \\
	\leq &  C\left[\left(\frac{\hat\mu (\Omega_0)}{2^m}\right)^{\frac{1}{n} }\frac{2^{\frac{m+1}{n}}}{\mu_u (K_t)^{\frac{1}{n} }}\right]\cdot |\Omega_m|  \\
	\leq & \left[C^{1+\frac{1}{n} }\frac{2^{\frac{1}{n} }}{s_0^{\frac{1}{n} }}\mu_u (K_t)^{\frac{1}{n^2}}\right]\cdot \hat\mu (\Omega_m)   
	\leq \frac{1}{2}\hat\mu (\Omega_m). 
\end{align*}
This implies that the set 
$$\hat\mu \left(\left\{x\in \Omega\big|\phi<-s_0-\sum_{j=1}^{\infty}2^{-\frac{j}{n}}\mu_u (K_t)^{\frac{1}{n}}.\right\}\right)=0.$$
Hence, 
\beqs
\|\phi\|_{L^{\infty}(\Omega, \mu_u)}&\leq& s_0+\sum_{j=1}^{\infty}2^{-\frac{j}{n}}\mu_u (K_t)^{\frac{1}{n} }\\
&=&2^{n+1}C^{n+1}\mu_u (K_t)^{\frac{1}{n}}+\frac{1}{2^{\frac{1}{n} }-1}\mu_u (K_t)^{\frac{1}{n} }\\
&\leq& C\mu_u (K_t)^{\frac{1}{n} }.
\eeqs
Therefore, we get
\[-\psi\le t^{\frac{1}{n-1}}\|\phi\|_{L^{\infty}}\le Ct^{\frac{1}{n-1}}\mu_u(K_t)^{\frac{1}{n}}.\] 
By \eqref{infty-2}, this implies
$t^{\frac{n}{n-1}}\mu_u(K_t)\le C$. $t^{\frac{n}{n-1}}\mu_u(K_t)\le C(n,\Lambda,\Omega)$.

\vskip 10pt

To prove $(ii)$, we suppose $1\le p<\frac{n}{n-1}$. Note that for any $q>1$ and $1<\beta<\infty$, there exists $C=C(\Lambda_q,\beta)$, such that 
$$\int_{\Omega_s}H_k[u]\log( 1+H_k[u])^{\frac{\beta n}{n+1}}\le C.$$
Let $x=H_k[u]$, $y=\lp\frac{\alpha}{2}\rp^{\frac{n\beta}{n+1}}\cdot\lp \frac{-\phi-s}{\|\phi+s\|_{\psho(\Omega_s)}}\rp^{\beta}$ in \eqref{Young}. By the inequality $(a+b)^{\frac{1}{\beta}}\le a^{\frac{1}{\beta}}+b^{\frac{1}{\beta }}$ for $\beta>1$, we can get
\begin{align}\label{MT-use}
& \left[ \int_{\Omega_s}\lp \frac{-\phi-s}{\|\phi+s\|_{\psho(\Omega_s)}}\rp^{\beta}\,\chi_{K_t}d\mu_u\right]^{\frac{1}{\beta}}\nonumber\\
\le & \lp\frac{2}{\alpha}\rp^{\frac{n}{n+1}}\cdot \left[ \int_{\Omega_s}\lp\frac{\alpha}{2}\rp^{\frac{n\beta}{n+1}}\cdot \lp \frac{-\phi-s}{\|\phi+s\|_{\psho(\Omega_s)}}\rp^{\beta}\,H_k[u]d\mu_0\right]^{\frac{1}{\beta}}\nonumber\\
\le & \left[ \int_{\Omega_s}\exp\ls \alpha\lp \frac{-\phi-s}{\|\phi+s\|_{\psho(\Omega_s)}}\rp^{\frac{n+1}{n}}\rs\,d\mu_0\right]^{\frac{1}{\beta}}  \nonumber\\[5pt]
& \ \ \ +C\lf \int_{\Omega_s}H_k[u]\log( 1+H_k[u])^{\frac{\beta n}{n+1}}\rf^{\frac{1}{\beta}}  \nonumber\\
\le & C(n,p,\Lambda_q,\Omega,\beta).
\end{align} 
Therefore, by \eqref{infty-1},  we obtain
\[ \int_{\Omega_s}(-\phi-s)(\ddc\phi)^n=\|\phi+s\|_{\psho(\Omega_s)}^{n+1}\leq \hat\mu\lp  \Omega_s\rp^{1-\frac{1}{\beta}}\cdot \|\phi+s\|_{\psho(\Omega_s)},\]
from which we deduce
\begin{align}
r\hat\mu(\Omega_{s+r})\le C\hat\mu(\Omega_s)^{1+\frac{\beta-n-1}{n\beta}},
\end{align}
where $C>0$ depends on $k$, $n$, $\Omega$, $\Lambda_q$ and $\beta$ for every $0<\beta<\infty$. Then by the same arguments, we get
$$t^{\frac{n\beta}{(\beta-n-1)(n-1)}}\mu_u(K_t)\le C(n,p,\Lambda_q,\Omega,\beta),\ \ \ \forall 0<\beta<\infty,$$
i.e., $\|G\|_{L_*^{\frac{n\beta}{(\beta-n-1)(n-1)}}(\Omega,\, d\mu)}\leq C$. For any $p<\frac{n}{n-1}$, we choose $\beta$ sufficiently large such that $\frac{n\beta}{(\beta-n-1)(n-1)}>p$. Hence, we have $\|G\|_{L^p(\Omega,d\mu)}\leq C\|G\|_{L_*^{\frac{n\beta}{(\beta-n-1)(n-1)}}(\Omega,\, d\mu)}\leq C$,
and thereby completing the proof of $(ii)$.
\end{proof}

\begin{rem}
As can be seen from the proof, by modifying the exponents in \eqref{infty-1}, for any $\beta>0$, there exists $C>0$ depends on  $\beta$ and $\|H_k[u]^{\frac{n}{k}}\|_{L^1\lp \log L^1\rp^{\beta n}(\Omega)}$, such that
$$t^{\frac{n\beta}{(n-1)\beta+1}}\mu_u(\ls G_{\Omega}>t\rs)\le C.$$
Here $\|f\|_{L^1\lp \log L^1\rp^{\beta n}(\Omega)}=\int_{\Omega}|f|\lp\log\lp 1+|f|\rp\rp^{\beta n}\,d\mu_0$.
\end{rem}

\vskip 20pt

\section{Proof of the Sobolev inequality}\label{sec-sob}

Once we have the decay estimate of the Green's function \eqref{w-lp}, the rest proof for the Sobolev inequality is similar to \cite{TW}, which deals with the case of real Monge-Amp\`ere equation. For completeness,
we include the details here.

\begin{proof}[Proof of Theorem \ref{Sobolev}]
Define 
\beq\label{mini-fun}
s^*=\inf\left\{\int_{\Om}H_k^{i\bar j}v_i\bar v_{j}\,\omega_0^n,\ v\in C_0^1(\bar\Omega), \ \int_\Omega F(u)H_k[u]\,\omega_0^n=1\right\},
\eeq
where $F(s)=\int_0^sf(t)\,dt$
and
\[f(t)=\begin{cases}
|t|^{p-1}& \text{if} \ |t|<K,\\
K^{p-1}& \text{if} \ |t|\geq K.
\end{cases}\]
$K$ is going to be sent to infinity. Let $v$ be the function where \eqref{mini-fun} attains the infimum. Therefore, $v$ satisfies
\beq\label{min-eq}
\begin{cases}
\mathcal L_{u,k}v&=-\hat\lambda f(v)H_k[u]  \ \ \ \text{in} \ \Omega,\\
\  \ \ \ \ v&=0 \ \ \ \ \  \text{on} \ \partial \Omega.
\end{cases}
\eeq
Here $\hat\lambda$ is the Lagrange multiplier associated to the minimization. It suffices to establish a lower bound estimate on $s^*$.
By the choice of $f$, we obtain
\beq\label{lam-s}
\hat\lambda\leq s^*\leq p\hat\lambda. 
\eeq
We divide the proof into three steps:

{\it Step 1(The first eigenvalue estimate):} Let $\mu_u=H_k[u]\,\omega_0^n$ be the complex $k$-Hessian measure associated to $u$. By Theorem \ref{lp-green}, there exists $\alpha>0$ depending on $p$ such that for any $t>0$,
\beq\label{green-decay}
\mu_u(\{z\in\Omega: G>t\})\leq \alpha t^{-\frac{p}{2}},
\eeq
where $1\leq p\le \frac{2n}{n-1}$.

For any open subset $U\subset\Omega$, let $\psi_1$ be the first eigenfunction and $\lambda_1(U)$ the first eigenvalue  of $\mathcal L_{u,k}$ with respect to $\mu_u$, namely
\begin{equation}\label{eigen-first}
\begin{cases}
	\mathcal L_{u,k}\psi_1=-\lambda_1(U)\psi_1H_k[u]&   \text{in} \ U,\\
	\  \ \ \ \ \psi_1=0 &\text{on} \ \partial U.
\end{cases}
\end{equation}
Let $G_U$ be the Green's function of $\mathcal L_{u,k}$ on $U$. By the maximum principle, $0\leq G_U(z, w)\leq G(z,w)$ for any $z, w\in U$.
Then we have
\[\psi_1(w)=\lambda_1(U)\int_U G_U(\cdot,w)\psi_1(\cdot)H_k[u]\,\omega_0^n.\]
Normalizing $\psi_1$ so that it $\|\psi_1\|_{L^\infty(U)}=1$. Choose $w$ to be the maximum of $\psi_1$.
Then  we obtain
\beq\nonumber
1\leq\lambda_1(U)\int_U G_U(,w) H_k[u]\,\omega_0^n.
\eeq
By \eqref{green-decay}, we have
\begin{eqnarray*}
\int_U G_U(,w) H_k[u]\,\omega_0^n&=&\int_0^{+\infty} \mu_u(\{z\in U: G_U>t\})\,dt\\
&\leq &\int_0^{+\infty} \min\{\mu_u(U), \alpha t^{-\frac{p}{2}}\}\,dt\\[5pt]
&=& T\mu_u(U)+\alpha\int_{T}^{+\infty}t^{-\frac{p}{2}}\,dt\\
&\leq& \lp\frac{1}{\frac{p}{2}-1}+1\rp\alpha^{\frac{2}{p}}[\mu_u(U)]^{1-\frac{2}{p}},
\end{eqnarray*}
where $T=(\frac{\alpha}{\mu_u(U)})^{\frac{2}{p}}$. It follows that for $p\le \frac{2n}{n-1}$
\beq\label{eigen-est}
c^*:=\inf_{U\subset \subset \Omega}\ls\lambda_1(U)[\mu_u(U)]^{1-\frac{2}{p}}\rs \geq \lp 1-\frac{2}{p}\rp\alpha^{-\frac{2}{p}}>0.
\eeq

\vskip 10pt

{\it Step 2(Estimate of the level set of $v$):} 
Denote $M=\|v\|_{L^\infty}$ and $\Omega_t=\{v>M-t\}$ for $t\leq M$. We claim that for $p>2$
\beq\label{lever-est}
\mu_u(\Omega_t)\geq \left(\frac{tc^*}{2s^*M^{p-1}}\right)^{\frac{2}{p-2}},  \ t\in (0,M).
\eeq
In fact, by \eqref{min-eq},
\begin{eqnarray*}
\lambda_1(\Omega_t)&\leq& \frac{\int_{\Om}H_k^{i\bar j}v_i\bar v_{j}\,\omega_0^n}{\int_{\Omega_t}(v-M+t)^2H_k[u]\,\omega_0^n}\\[4pt]
&=& \frac{\int_{\Om} (v-M+t)(-\mathcal L_{u,k}v)\,\omega_0^n}{\int_{\Omega_t}(v-M+t)^2H_k[u]\,\omega_0^n}\\[4pt]
&=& s^*M^{p-1}\frac{\int_{\Om} (v-M+t)H_k[u]\,\omega_0^n}{\int_{\Omega_t}(v-M+t)^2H_k[u]\,\omega_0^n}\\[4pt]
&\leq & s^*M^{p-1}\left(\frac{\mu_u(\Omega_t)}{\int_{\Omega_t}(v-M+t)^2H_k[u]\,\omega_0^n}\right)^{\frac{1}{2}}\\
&\leq & s^*M^{p-1}\left(\frac{\mu_u(\Omega_t)}{(\frac{t}{2})^2\mu_u(\Omega_{t/2})}\right)^{\frac{1}{2}}.
\end{eqnarray*}
By \eqref{eigen-est},
\[c^*\mu_u(\Omega_t)^{\frac{2}{p}-1}\leq\lambda_1(\Omega_t)\leq\frac{2s^*M^{p-1}}{t}\left(\frac{\mu_u(\Omega_t)}{\mu_u(\Omega_{t/2})}\right)^{\frac{1}{2}},\]
i.e.,
\beq\label{lever-decay}
\mu_u(\Omega_t)\geq \beta\mu_u(\Omega_{t/2})^{\frac{p}{3p-4}}, \ \ \beta=\left(\frac{tc^*}{2s^*M^{p-1}}\right)^{\frac{2p}{3p-4}}.
\eeq
By iteration, for any $m$,
\[
\mu_u(\Omega_t)\geq \beta^{\sum_{j=0}^{m-1}(\frac{p}{3p-4})^j}\mu_u(\Omega_{t/2^m})^{(\frac{p}{3p-4})^m}.
\]
It is easy to see that $$\sum_{j=0}^{+\infty}(\frac{p}{3p-4})^j=\frac{3p-4}{2(p-2)}$$ and hence 
\[\beta^{\sum_{j=0}^{+\infty}(\frac{p}{3p-4})^j}=\beta^{\frac{3p-4}{2(p-2)}}=\left(\frac{tc^*}{2s^*M^{p-1}}\right)^{\frac{2}{p-2}}.\]
It suffices to show that
\beq\label{lim-term}
\lim_{m\to+\infty}\mu_u(\Omega_{t/2^m})^{(\frac{p}{3p-4})^m}=1.
\eeq
Indeed, we choose $z_0$ such that $v(z_0)=M$ and $a=\|\nabla v\|_{L^\infty(\Omega)}$. Then we have $B(z_0,\frac{t}{2^ma})\subset\Omega_{t/2^m}$. By condition \eqref{low-upp},
\begin{eqnarray}\label{db-arg}
\mu_u(\Omega_{t/2^m})\geq \mu_u(B(z_0,\frac{t}{2^ma}))&\geq&\lambda |B(z_0,\frac{t}{2^ma})|\nonumber\\
&=&\frac{\lambda}{2^{2n}} |B(z_0,\frac{t}{2^{m+1}a})|\nonumber\\
&=&\frac{\lambda}{2^{2n}\Lambda} \mu_u(B(z_0,\frac{t}{2^{m+1}a}))\nonumber\\
&\geq& \left(\frac{\lambda}{2^{2n}\Lambda}\right)^{m-m_0} \mu_u(B(z_0,\frac{t}{2^{m_0}a})),
\end{eqnarray}
where $m_0$ is the smallest number such that $B(z_0,\frac{t}{2^{m_0}a})\subset\Omega$. Therefore, \eqref{lim-term} holds when $p>2$.

\vskip 10pt

{\it Step 3(The lower bound of $s^*$):} 
We claim that 
$$s^*\geq C c^*.$$
By the normalizing condition $\int_\Omega F(v)H_k[u]\,\omega_0^n=1$, we have
\begin{eqnarray*}
1&=&\int_0^{F(M)}\mu_u(\{z\in\Omega:\ F(v)>t\})\,dt\\
&=&\int_0^{M}\mu_u(\{z\in\Omega:\ v>t\})F'(t)\,dt\\
&=&\int_0^{M}\mu_u(\{z\in\Omega:\ v>M-t\})F'(M-t)\,dt\\
&=&\int_0^{M}\mu_u(\Omega_t)F'(M-t)\,dt\\
&\geq &\left(\frac{c^*}{2s^*M^{p-1}}\right)^{\frac{2}{p-2}}\int_0^{M}t^{\frac{2}{p-2}}F'(M-t)\,dt\\
&=&\left(\frac{c^*}{2s^*}\right)^{\frac{2}{p-2}}\int_0^1t^{\frac{2}{p-2}}\frac{F'(M(1-t))}{M^{p-1}}\,dt.
\end{eqnarray*}
It remains to prove
\beq\label{int-est1}
\int_0^1t^{\frac{2}{p-2}}\frac{F'(M(1-t))}{M^{p-1}}\,dt\geq C>0.
\eeq
By the definition of $f$,
\[F'(M(1-t))=\begin{cases}
pM^{p-1}(1-t)^{p-1},&\ \text{if}\ \ 1-\frac{K}{M}<t<1,\\[4pt]
K^{p-1},&\ \text{if }\ \ 0<t<1-\frac{K}{M}.
\end{cases}\]
Hence, we have
\beq\label{int-F}
\int_0^1t^{\frac{2}{p-2}}\frac{F'(M(1-t))}{M^{p-1}}\,dt=\left(\frac{K}{M}\right)^{p-1}\int_0^{1-\frac{K}{M}}t^{\frac{2}{p-2}}\,dt+p\int_{1-\frac{K}{M}}^1t^{\frac{2}{p-2}}(1-t)^{p-1}\,dt.
\eeq
Let $V=\{z\in\Omega|\ v>K\}$.
We use 
\[1=\int_\Omega F(v)H_k[u]\,\omega_0^n\geq \frac{1}{p}\int_V K^pH_k[u]\,\omega_0^n\]
again to get $\mu_u(V)\leq pK^{-p}$.
In $V$, $v$ satisfies the equation 
\[\mathcal L_{u,k}v=-\hat\lambda f(v)H_k[u].\]
By \eqref{lam-s},
$\hat\lambda\leq s^*\leq C$. Then by \eqref{eigen-est} with $v=K$ on $\partial V$ we have
\begin{eqnarray*}
v(z)&\leq& K+\hat\lambda K^{p-1}\int_V G_V(z,\cdot)H_k[u]\,\omega_0^n\\
&\leq& K+C K^{p-1}\mu_u(V)^{1-\frac{2}{p}}
\leq (1+C)K.
\end{eqnarray*} 
That is, $\frac{K}{M}\geq \frac{1}{1+C}$. In conclusion, by \eqref{int-F}, we obtain \eqref{int-est1}.
\end{proof}

\begin{rem} 
By checking the proof, the  Sobolev inequalities hold under $H_k[u]\leq \Lambda$ and the doubling condition 
\begin{align}\label{doub}
\mu_u(B_r)\ge b\mu_u(B_{2r}),\ \ \ \text{for }0<r<1\text{, and some }b>0.
\end{align}  
Indeed, the only place using the lower bound of $H_k[u]$ is \eqref{db-arg}, which can be replaced by
\begin{eqnarray*}
\mu_u(\Omega_{t/2^m})\geq \mu_u(B(z_0,\frac{t}{2^ma}))
\geq b \cdot\mu_u(B(z_0,\frac{t}{2^{m-1}a}) 
\geq b^{m-m_0} \mu_u(B(z_0,\frac{t}{2^{m_0}a})).
\end{eqnarray*}
\end{rem}

\vskip 20pt

\section{A local estimate for complex Hessian equations}\label{sec:local}

In section, we study the local regularity of $W^{2, p}$-solutions of the complex Hessian equations under suitable regularity assumptions on 
the right-hand side.


\begin{proof}[Proof of Theorem \ref{main2}]
In \cite{DK}, the same estimate was obtained under the assumption $f\in C^{1,1}$. Here we only need $f\in C^1$.
Denote $v=\triangle u$. Differentiating \eqref{k-hess-c} twice, we have
\begin{eqnarray*}
\frac{1}{H_k[u]}H_k^{i\bar j}u_{i\bar jl\bar l}\geq g_{l\bar l}=\frac{f_{l\bar l}}{f}-\frac{|f_l|^2 }{f^2},
\end{eqnarray*}
i.e., 
\[H_k^{i\bar j}v_{i\bar j}\geq \Delta f-\frac{|\D f|^2}{f}.\]
Hence, we have
\begin{equation}\label{eq-41}
\int_{\Omega} H_k^{i\bar j}D_{i}v D_{\bar j}\varphi\leq \int_{\Omega} \lp \frac{|\D f|^2}{f}-\Delta f\rp \varphi\,d\mu_0\quad \text{for\,\,any}\,\, \varphi\in W^{1,2}_0(\Omega)\text{\ and}\,\,\varphi\geq 0.
\end{equation}

Let $\eta\in C^\infty_0(B_1)$ and choose the test function
$\varphi=\eta^2v^{\beta+1}\in W^{1,p}_0(B_1)$
for some $\beta\geq 0$ to be determined later. By \eqref{eq-41},
\begin{eqnarray*}
(\beta+1)\int_{\Omega}H_k^{i\bar j}D_iv D_{\bar j}v v^{\beta} \eta^2\, d\mu_0
\leq 2\int_{\Omega} |H_k^{i\bar j} D_iv D_{\bar j}\eta| v^{\beta+1} \eta \, d\mu_0+\int_{\Omega}\lp \frac{|\D f|^2}{f}-\Delta f\rp v^{\beta+1}\eta^2\, d\mu_0,
\end{eqnarray*}
which implies
\begin{eqnarray*}
(\beta+\frac{1}{2})\int_{\Omega}H_k^{i\bar j}D_iv D_{\bar j}v v^{\beta} \eta^2\, d\mu_0
\leq 2\int_{\Omega} H_k^{i\bar j} D_i\eta D_{\bar j}\eta v^{\beta+2}  \, d\mu_0+\int_{\Omega}\lp \frac{|\D f|^2}{f}-\Delta f\rp v^{\beta+1}\eta^2\, d\mu_0.
\end{eqnarray*}
By direct computations, we have
\begin{eqnarray*}
&& \int_{\Omega}\lp \frac{|\D f|^2}{f}-\Delta f\rp v^{\beta+1}\eta^2\,d\mu_0 \\ 
&\le& C\int_{\Omega}|D f|^2v^{\beta+1}\eta^2\,d\mu_0+(\beta+1)\int_{\Omega}v^{\beta}\lan \D f,\,\D v\ran\eta^2\,d\mu_0   +\int_{\Omega}2\eta\lan \D f,\D\eta\ran v^{\beta+1}\,d\mu_0  \\
&\le & C\int_{\Omega}v^{\beta+1}\eta^2(1+|\D \eta|^2)\,d\mu_0+2C(\beta+1)\int_{\Omega}\sum_i\frac{1}{H_{k}^{i\bar i}}\left|\D f\right|^2v^{\beta}\eta^2\,d\mu_0  \\
&& \ \ \ +\frac{\beta+1}{2}\int_{\Omega}\sum_iH_k^{i\bar i}D_iv D_{\bar i}v \,v^{\beta}\eta^2\,d\mu_0.   
\end{eqnarray*}
Note that $0<\lambda\leq H_k[u]\leq \Lambda$. 
Recalling the fundamental inequalities \cite{W}, we have
$$\prod_i H_k^{i\bar i}\ge C_{n,k}{ H_{k-1}}^{\frac{n(k-2)}{k-1}}\ge C.$$
Hence, 
$$\sum_i\frac{1}{H_k^{i\bar i}}=\frac{\sum_i\prod_{j\neq i}H_k^{j\bar j}}{\prod_iH_k^{i\bar i}}\le 
C\frac{H_k[u]^{n-1}}{\prod_iH_k^{i\bar i}}\leq C.$$
We conclude that
\begin{align}\label{eq-42}
\beta\int_{\Omega}v^{\beta}H_k^{i\bar j}D_iv D_{\bar j}v  \eta^2\, d\mu_0\le 4\int_{\Omega} H_k^{i\bar j} D_i\eta D_{\bar j}\eta v^{\beta+2}  \, d\mu_0+C(\beta+1)\int_{\Omega}v^{\beta+1}\eta^2\,d\mu_0.
\end{align}
Let $w=v^{\frac{\beta}{2}+1}$. Then $D_iw=(\frac{\beta}{2}+1)v^{\frac{\beta}{2}} v_i$
and
\begin{eqnarray*}
H_k^{i\bar j}D_i(w\eta)D_{\bar j}(w\eta)&\le& 2H_k^{i\bar j}D_iw D_{\bar j}w\eta^2 + 2H_k^{i\bar j}\D_i\eta D_{\bar j} w^2\\
&= &2\lp\frac{\beta}{2}+1\rp^2v^{\beta}H_k^{i\bar j}D_iv D_{\bar j}v  \eta^2+2v^{\beta+2}H_k^{i\bar j}D_i\eta D_{\bar j}\eta.
\end{eqnarray*}
By $\sum_iH^{i\bar i}_k=C_{n,k} H_{k-1}[u]\leq C'_{n,k} v^{k-1}$ and \eqref{eq-42},
\begin{eqnarray}
&&\int_{B_1}H_k^{i\bar j}D_i(w\eta)D_{\bar j}(w\eta)\,d\mu_0\nonumber\\
&\leq&
\int_{B_1}2(\frac{\beta}{2}+1)^2H_k^{i\bar j}D_iv D_{\bar j}v v^{\beta} \eta^2\, d\mu_0
+\int_{B_1}2v^{\beta+2}H_k^{i\bar j}D_i\eta D_{\bar j}\eta\,d\mu_0\nonumber\\
&\leq&C(\beta+1)\left(\int_{B_1}v^{\beta+2}H_k^{i\bar j}D_i\eta D_{\bar j}\eta\,d\mu_0+(\beta+1)\int_{B_1}v^{\beta+1}\eta^2\, d\mu_0\right).\nonumber\\
&\leq&C(\beta+1)^2\int_{B_1}v^{\beta+1+k}|D\eta|^2\,d\mu_0.\nonumber
\end{eqnarray}
Here we also used $\triangle v\geq H_k[u]^{\frac{1}{k}}\geq C$.
Hence, by the Sobolev inequality \eqref{sob-ineq},
\begin{eqnarray}
\|w\eta\|_{L^\frac{2n}{n-1}(B_1)}&\leq &C\left(\int_{B_1}H_k^{i\bar j}D_i(w\eta)D_{\bar j}(w\eta)\,d\mu_0\right)^{\frac{1}{2}}\nonumber\\
&\leq&C(\beta+1)\left(\int_{B_1}v^{\beta+1+k}|D\eta|^2\,d\mu_0\right)^{\frac{1}{2}}.\nonumber
\end{eqnarray}
Now for any $0<r<R\leq 1$, we choose the  cutoff function $\eta\in C_0^\infty(B_R)$ with the property
$$0\leq\eta\leq 1,\quad\eta\equiv 1\,\,\text{in}\,\, B_r\quad\text{and}\quad |D\eta|\leq \frac{2}{R-r}.$$
Then we obtain
$$\|v^{\frac{\beta}{2}+1}\|_{L^\frac{2n}{n-1}(B_r)}\leq \frac{C(1+\beta)}{R-r}\left(\int_{B_R}v^{\beta+1+k}\,d\mu_0\right)^{\frac{1}{2}}.$$
The exponent of  integrability is improved if
 $(\frac{\beta}{2}+1)\frac{2n}{n-1}>\beta+1+k$,
i.e.,
$\beta>(n-1)k-n-1$.

Let $\alpha=\beta+2$, $\gamma=\frac{n}{n-1}$ and $\rho=R-r$.
$$\left(\int_{B_r}v^{\alpha\gamma}\,d\mu_0\right)^{\gamma^{-1}}\leq \frac{C\alpha^2}{\rho^2}\left(\int_{B_R}v^{\alpha-1+k}\,d\mu_0\right).$$
We now define a sequence $\{\alpha_j\}$ inductively by
$$\alpha_j=\gamma\alpha_{j-1}-(k-1).$$
By iteration,
\begin{eqnarray*}
\left(\int_{B_{r+2^{-j}\rho}}v^{\alpha_j\gamma}\,d\mu_0\right)^{\gamma^{-(j+1)}}\leq \left(\frac{C}{\rho^2}\right)^{\sum_{i=0}^j\gamma^{-i}}\cdot 4^{\sum_{i=0}^j i\gamma^{-i}}\cdot\prod_{i=0}^j \alpha_i^{2\gamma^{-i}}\cdot\left(\int_{B_R}v^{\alpha_0-1+k}\,d\mu_0\right).
\end{eqnarray*}
Since $\alpha_i=\gamma\alpha_{i-1}-(k-1)\leq \gamma\alpha_{i-1}\leq \cdots\leq \gamma^i\alpha_0$, we have
\[\prod_{i=0}^j \alpha_i^{2\gamma^{-i}}\leq \prod_{i=0}^j (\gamma^i\alpha_0)^{2\gamma^{-i}}=\gamma^{\sum_{i=0}^j2i\gamma^{-i}}\cdot
\alpha_0^{\sum_{i=0}^j2\gamma^{-i}}\leq C\]
as $j\to\infty$.
By direct computation,
\begin{eqnarray*}
\alpha_j&=&\gamma\alpha_{j-1}-(k-1)\\
&=&\cdots\\
&=&\gamma^{j}\alpha_0-(k-1)(1+\gamma+\cdots+\gamma^{j-1})\\
&=&\gamma^{j}\alpha_0-(k-1)\frac{1-\gamma^j}{1-\gamma}\\
&=&\gamma^{j}(\alpha_0-\frac{k-1}{\gamma-1})+\frac{k-1}{\gamma-1}.
\end{eqnarray*}
Therefore,
\[\left(\int_{B_{r+2^{-j}\rho}}v^{\alpha_j\gamma}\,d\mu_0\right)^{\frac{1}{\alpha_j\gamma} \left[\alpha_0-(k-1)\frac{1-\gamma^{-j}}{\gamma-1}\right]  }\leq C\int_{B_R}v^{\alpha_0-1+k}\,d\mu_0.\]
It is clear that $\alpha_j\to\infty$ when $\alpha_0>\frac{k-1}{\gamma-1}=(k-1)(n-1)$, i.e., 
$\alpha_0+k-1>(k-1)n$. This implies 
$$\sup_{B_r}|\triangle u|\le C\left(\int_{B_R}v^{p}\,d\mu_0\right)^{\frac{1}{p-(k-1)n}}$$
for  $C>0$ depends on $r$, $R$,
$k$, $n$, $p$, $\lambda$, $\Lambda$, $\|f\|_{C^1(\Omega)}$ and $\|u\|_{W^{2,p}(\Omega)}$ with $p>(k-1)n$.
\end{proof}

\vskip 20pt

\section{The Harnack inequality}\label{sec-har}

With the Sobolev inequality, a natural idea is to use De Giorgi-Nash-Moser's iteration method to investigate the regularity of the equations.
Balls in this section will be generally assumed concentric, in which case we simply write $B_{R}(z_0)=B_R$.

\subsection{The local boundedness}

First, we show the local boundedness of the solution to the linearized complex Hessian equation \eqref{com-div-1}.  For later use, we consider the equation with a lower order term
\beq\label{com-div-2}
D_{\bar j}(H_k^{i\bar j}D_i v)+c(z)v(z)=f(z).
\eeq

\begin{theo}[Local boundedness]\label{thm-har}
Assume $u$ is a smooth, strictly $k$-plurisubharmonic function($2\leq k\leq n$) that satisfies 
\eqref{low-upp}
for some $\lambda, \Lambda\in\mathbb R$. Assume $f(z), c(z)\in L^q(\Omega)$ with $q>n$.
Let $v\in W^{1,2}(\Omega)$ be a subsolution to 
\eqref{com-div-2}
in the following sense:
\begin{equation}\label{eq}
\int_{\Omega} H_k^{i\bar j}D_{i}v D_{\bar j}\varphi+cv\varphi\, d\mu_0\leq \int_{\Omega} f\varphi\,d\mu_0\quad \text{for\,\,any}\,\, \varphi\in W^{1,2}_0(\Omega)\text{\ and}\,\,\varphi\geq 0.
\end{equation}
Then for any ball $B_R\subset \Omega$, $0<r<R$, $n<q_0\leq q$, and $p>0$
\beq\label{loc-b}
\sup_{B_{r}}v\leq C \Theta_{q_0}(B_R)^{\frac{\delta}{4}}\cdot\left(\frac{1}{(R-r)^{\frac{2n}{p}}}\|v^+\|_{L^p(B_R)}+R^{2-\frac{2n}{q}}\|f\|_{L^{q}(B_R)}\right),\eeq
where
where $\delta=\max\{\frac{q_0n}{q_0-n}, \frac{qn}{q-n}\}$ and $C$ depends only on $k$, $n$, $p$, $q_0$, $q$,  $\lambda$, $\Lambda$,  $\|f\|_{L^q(\Omega)}$ and $\|c\|_{L^q(\Omega)}$.
\end{theo}

\begin{proof}
We first consider $\Omega=B_1$, $p=1$, $R=1$ and $r=\frac{1}{2}$. 

For some $K>0$ and $m>0$, set $\tilde{v}=v^++K$ and
$$\tilde{v}_m=\left\{\begin{array}{lc}\tilde{v},&\ v<m,\\ K+m,&\ v\geq m.\end{array}\right.$$
Let $\eta\in C^\infty_0(B_1)$. We choose the test function
$\varphi=\eta^2(\tilde{v}_m^\beta \tilde{v}-K^{\beta+1})\in W^{1,p}_0(B_1)$
for some $\beta\geq 0$ to be determined later. Substituting $\varphi$ into \eqref{eq}, we have
\begin{eqnarray*}
&&\beta\int_{B_1}H_k^{i\bar j}D_i\tilde{v}_m D_{\bar j}\tilde{v}_m \tilde{v}_m^{\beta} \eta^2\, d\mu_0+\int_{B_1}H_k^{i\bar j}D_i\tilde{v} D_{\bar j}\tilde{v}\,\tilde{v}_m^\beta \eta^2\, d\mu_0\\
&=&\beta\int_{B_1}H_k^{i\bar j}D_i\tilde{v} D_{\bar j}\tilde{v}_m \tilde{v}_m^{\beta-1} \tilde{v}\eta^2\, d\mu_0+\int_{B_1}H_k^{i\bar j}D_i\tilde{v} D_{\bar j}\tilde{v}\,\tilde{v}_m^\beta \eta^2\, d\mu_0\\[4pt]
&\leq& 2\int_{B_1} |H_k^{i\bar j} D_i\tilde{v} D_{\bar j}\eta|  \tilde{v}_m^\beta \tilde{v}  \eta \, d\mu_0+\int_{B_1} \left(|c|\eta^2 \tilde{v}_m^\beta \tilde{v}^2+|f|\eta^2 \tilde{v}_m^\beta \tilde{v}\right)\, d\mu_0.
\end{eqnarray*}
Note that $D\tilde{v}=D\tilde{v}_m$ in $\{v<m\}$ and $D\tilde{v}_m=0$ in $\{v\geq m\}$. 
By Cauchy's inequality, we have
\[ 2\int_{B_1} |H_k^{i\bar j} D_i\tilde{v} D_{\bar j}\eta| \tilde{v}_m^\beta \tilde{v}  \eta \, d\mu_0\leq \frac{1}{2}\int_{B_1} H_k^{i\bar j} D_i\tilde{v} D_{\bar j}\tilde{v} \, \tilde{v}_m^\beta \eta^2\, d\mu_0+4\int_{B_1} H_k^{i\bar j}D_i\eta D_{\bar j} \eta \tilde{v}_m^\beta \tilde{v}^2\, d\mu_0.\]
It follows
\begin{eqnarray}
&&\beta\int_{B_1}H_k^{i\bar j}D_i\tilde{v}_m D_{\bar j}\tilde{v}_m \tilde{v}_m^{\beta} \eta^2\, d\mu_0+\frac{1}{2}\int_{B_1}H_k^{i\bar j}D_i\tilde{v} D_{\bar j}\tilde{v}\, \tilde{v}_m^\beta \eta^2\, d\mu_0\nonumber\\[4pt]
&\leq& 4\int_{B_1} H_k^{i\bar j}D_i\eta D_{\bar j} \eta \tilde{v}_m^\beta \tilde{v}^2\, d\mu_0+\int_{B_1} \left(|c|\eta^2 \tilde{v}_m^\beta \tilde{v}^2+|f|\eta^2 \tilde{v}_m^\beta \tilde{v}\right)\, d\mu_0.\label{ineq1}
\end{eqnarray}
Let $w=\tilde{v}_m^{\frac{\beta}{2}}\tilde{v}$. Then 
\[D_iw=\frac{\beta}{2}\tilde{v}_m^{\frac{\beta}{2}-1}\tilde{v} D_i\tilde{v}_m+\tilde{v}_m^{\frac{\beta}{2}}D_i\tilde{v}=\frac{\beta}{2}\tilde{v}_m^{\frac{\beta}{2}} D_i\tilde{v}_m+\tilde{v}_m^{\frac{\beta}{2}}D_i\tilde{v}\]
and
\begin{eqnarray*}
H_k^{i\bar j}D_i(w\eta)D_{\bar j}(w\eta)&=&\eta^2H_k^{i\bar j}(\frac{\beta}{2}\tilde{v}_m^{\frac{\beta}{2}} D_i\tilde{v}_m+\tilde{v}_m^{\frac{\beta}{2}}D_i\tilde{v})(\frac{\beta}{2}\tilde{v}_m^{\frac{\beta}{2}} D_{\bar j}
\tilde{v}_m+\tilde{v}_m^{\frac{\beta}{2}}D_{\bar j}\tilde{v})+ w^2H_k^{i\bar j}D_i\eta D_{\bar j}\eta\\
&&+\eta wH_k^{i\bar j}D_i\eta D_{\bar j}w+\eta wH_k^{i\bar j}D_{\bar j}\eta D_iw\\
&=&\eta^2\tilde{v}_m^{\beta}[(\frac{\beta^2}{4}+\beta) H_k^{i\bar j}D_i\tilde{v}_m D_{\bar j}
\tilde{v}_m+H_k^{i\bar j}D_i\tilde{v}D_{\bar j}\tilde{v}]+ \tilde{v}_m^\beta \tilde{v}^2H_k^{i\bar j}D_i\eta D_{\bar j}\eta\\
&&+\eta wH_k^{i\bar j}D_i\eta (\frac{\beta}{2}\tilde{v}_m^{\frac{\beta}{2}} D_{\bar j}\tilde{v}_m+\tilde{v}_m^{\frac{\beta}{2}}D_{\bar j}\tilde{v})+\eta wH_k^{i\bar j}D_{\bar j}\eta (\frac{\beta}{2}\tilde{v}_m^{\frac{\beta}{2}} D_i\tilde{v}_m+\tilde{v}_m^{\frac{\beta}{2}}D_i\tilde{v})\\
&\leq &\eta^2\tilde{v}_m^{\beta}[(\frac{3\beta^2}{4}+\beta) U^{i\bar j}D_i\tilde{v}_m D_{\bar j}
\tilde{v}_m+2H_k^{i\bar j}D_i\tilde{v}D_{\bar j}\tilde{v}]+4 \tilde{v}_m^\beta \tilde{v}^2H_k^{i\bar j}D_i\eta D_{\bar j}\eta.
\end{eqnarray*}
Hence, by \eqref{ineq1}, we have
\begin{eqnarray}
&&\int_{B_1}H_k^{i\bar j}D_i(w\eta)D_{\bar j}(w\eta)\,d\mu_0\nonumber\\
&\leq&4(\beta+1)
\left[\beta\int_{B_1}H_k^{i\bar j}D_i\tilde{v}_m D_{\bar j}\tilde{v}_m \tilde{v}_m^{\beta} \eta^2\, d\mu_0+\frac{1}{2}\int_{B_1}H_k^{i\bar j}D_i\tilde{v} D_{\bar j}\tilde{v}\, \tilde{v}_m^\beta \eta^2\, d\mu_0\right]\nonumber\\
&&+
\int_{B_1}4 \tilde{v}_m^\beta \tilde{v}^2H_k^{i\bar j}D_i\eta D_{\bar j}\eta\,d\mu_0\nonumber\\
&\leq&4(4\beta+5)
\int_{B_1}H_k^{i\bar j}D_i\eta D_{\bar j}\eta\tilde{v}_m^\beta \tilde{v}^2\, d\mu_0+4(\beta+1)\int_{B_1} \left(|c|\eta^2 \tilde{v}_m^\beta \tilde{v}^2+|f|\eta^2 \tilde{v}_m^\beta \tilde{v}\right)\, d\mu_0\nonumber\\
&\leq& 4(4\beta+5)
\int_{B_1}\mathcal U |D\eta|^2 \tilde{v}_m^\beta \tilde{v}^2\, d\mu_0+4(\beta+1)\int_{B_1} c_0\eta^2 \tilde{v}_m^\beta \tilde{v}^2\, d\mu_0,\label{ineq2}
\end{eqnarray}
where $c_0=|c|+\frac{|f|}{K}$.
Choose $K=\|f\|_{L^q(B_1)}$ if $f$ is not identically 0. Otherwise choose arbitrary $K>0$ and let $K\to 0^+$.
By H{\"o}lder's inequality, we have
\begin{eqnarray}
\int_{B_{1}} \mathcal U|D\eta|^2w^2\,dx&\leq&\|\mathcal U\|_{L^{q_0}(B_1)}\cdot\|wD\eta\|_{L^{\frac{2q_0}{q_0-1}}(B_1)}^2,\label{norm-est-1}\\[3pt]
\int_{B_1}c_0w^2\eta^2\,dx&\leq&\|c_0\|_{L^{q}(B_1)}\cdot\left(\int_{B_1}(\eta w)^{\frac{2q}{q-1}}\,dx\right)^{1-\frac{1}{q}}.\label{norm-est-3}
\end{eqnarray}

Assume $q\geq q_0$.
Combining \eqref{ineq2}-\eqref{norm-est-3},  we get
\begin{eqnarray*}
\left(\int_{B_1}H_k^{i\bar j}D_i(w\eta)D_{\bar j}(w\eta)\,d\mu_0\right)^\frac{1}{2}&\leq& C\|\mathcal U\|_{L^{q_0}(B_1)}^{\frac{1}{2}}(1+\beta)^{\frac{1}{2}}\left(\|wD\eta\|_{L^{\frac{2q_0}{q_0-1}}(B_1)}+\|w\eta\|_{L^{\frac{2q}{q-1}}(B_1)}\right)\\[4pt]
&\leq& C\|\mathcal U\|_{L^{q_0}(B_1)}^{\frac{1}{2}}(1+\beta)^{\frac{1}{2}}\left(\|wD\eta\|_{L^{\frac{2q_0}{q_0-1}}(B_1)}+\|w\eta\|_{L^{\frac{2q_0}{q_0-1}}(B_1)}\right)
\end{eqnarray*}
for $C$ depending on $\lambda$, $\Lambda$, $q_0$, $q$ and $\|c\|_{L^q(\Omega)}$. 
By Sobolev inequality \eqref{sob-ineq},
\begin{equation}\label{re-Holder}
\|w\eta\|_{L^\frac{2n}{n-1}(B_1)}\leq C\|\mathcal U\|_{L^{q_0}(B_1)}^{\frac{1}{2}}(1+\beta)^{\frac{1}{2}}\left(\|wD\eta\|_{L^{\frac{2q_0}{q_0-1}}(B_1)}+\|w\eta\|_{L^{\frac{2q_0}{q_0-1}}(B_1)}\right),
\end{equation}
If $q<q_0$, we can replace $\frac{2q_0}{q_0-1}$ by $\frac{2q}{q-1}$ and the following iteration argument will be similar.

Now for any $0<r<R\leq 1$, we choose the  cutoff function $\eta\in C_0^\infty(B_R)$ with the property
$$0\leq\eta\leq 1,\quad\eta\equiv 1\,\,\text{in}\,\, B_r\quad\text{and}\quad |D\eta|\leq \frac{2}{R-r}.$$
Then we obtain
$$\|w\|_{L^\frac{2n}{n-1}(B_r)}\leq \frac{C\|\mathcal U\|_{L^{q_0}(B_1)}^{\frac{1}{2}}(1+\beta)^{\frac{1}{2}}}{R-r}\|w\|_{L^{\frac{2q_0}{q_0-1}}(B_R)}.$$
We can do the iteration as follows.

Recalling the definition of $w$, we have
$$\|\tilde{v}_m^{\frac{\beta}{2}}\tilde{v}\|_{L^\frac{2n}{n-1}(B_r)}\leq\frac{C\|\mathcal U\|_{L^{q_0}(B_1)}^{\frac{1}{2}}(1+\beta)^{\frac{1}{2}}}{R-r}\|\tilde{v}_m^{\frac{\beta}{2}}\tilde{v}\|_{L^{\frac{2q_0}{q_0-1}}(B_R)}.$$
Set $\gamma=\beta+2\geq 2$. 
By $\tilde{v}_m\leq \tilde{v}$, we obtain
$$\|\tilde{v}_m^{\frac{\gamma}{2}}\|_{L^\frac{2n}{n-1}(B_r)}\leq\frac{C\|\mathcal U\|_{L^{q_0}(B_1)}^{\frac{1}{2}}\gamma^{\frac{1}{2}}}{R-r}\|\tilde{v}^{\frac{\gamma}{2}}\|_{L^{\frac{2q_0}{q_0-1}}(B_R)}.$$
Letting $m\to\infty$, we get
$$\|\tilde{v}^{\frac{\gamma}{2}}\|_{L^\frac{2n}{n-1}(B_r)}\leq\frac{C\|\mathcal U\|_{L^{q_0}(B_1)}^{\frac{1}{2}}\gamma^{\frac{1}{2}}}{R-r}\|\tilde{v}^{\frac{\gamma}{2}}\|_{L^{\frac{2q_0}{q_0-1}}(B_R)},$$
i.e.,
$$\|\tilde{v}\|_{L^{{\frac{\gamma}{2}}\frac{2n}{n-1}}(B_r)}\leq\frac{(C\|\mathcal U\|_{L^{q_0}(B_1)}^{\frac{1}{2}}\gamma)^{\frac{1}{\gamma}}}{(R-r)^{\frac{2}{\gamma}}}\|\tilde{v}\|_{L^{{\frac{\gamma}{2}}\frac{2q_0}{q_0-1}}(B_R)}.$$
Denote $\chi=\frac{2n}{n-1}\cdot\frac{q_0-1}{2q_0}>1$. Then
\begin{equation}\label{itt}
\|\tilde{v}\|_{L^{{\frac{{\gamma}q_0}{q_0-1}\chi}}(B_r)}\leq\frac{(C\|\mathcal U\|_{L^{q_0}(B_1)}^{\frac{1}{2}}\gamma)^{\frac{1}{\gamma}}}{(R-r)^{\frac{2}{\gamma}}}\|\tilde{v}\|_{L^{{\frac{{\gamma}q_0}{q_0-1}}}(B_R)}.
\end{equation}
We iterate \eqref{itt} to get the desired estimate. Set
$$\gamma_i=2\chi^{i}\quad\text{and}\quad R_i=r+\frac{R-r}{2^{i}},\  i=0,1,2,\cdots,$$
i.e.,
$$\gamma_i=\chi\gamma_{i-1}\ \text{and} \ R_{i-1}-R_i=\frac{R-r}{2^{i}},\  i=1,2,\cdots.$$
By \eqref{itt},
$$\|\tilde{v}\|_{L^{\frac{2q_0}{q_0-1}\chi^{i+1}}(B_{R_{i+1}})}
\leq (C\|\mathcal U\|_{L^{q_0}(B_1)}^{\frac{1}{2}})^{\sum_{j=0}^i\frac{1}{\gamma_j}}\cdot\prod_{j=0}^i \gamma_j^{\frac{1}{\gamma_j}}\cdot 4^{\sum_{j=0}^i\frac{j}{\gamma_j}}\frac{1}{(R-r)^{\sum_{j=0}^i\frac{2}{\gamma_j}}}\cdot\|\tilde{v}\|_{L^{\frac{2q_0}{q_0-1}}(B_R)}.$$
Letting $i\to\infty$, by Young's inequality, we have
$$\begin{aligned}
\|\tilde{v}\|_{L^\infty(B_r)}&\leq \frac{C\|\mathcal U\|_{L^{q_0}(B_1)}^{\frac{1}{4}\frac{\chi}{\chi-1}}}{(R-r)^{\frac{\chi}{\chi-1}}}\|\tilde{v}\|_{L^{\frac{2q_0}{q_0-1}}(B_R)}\\[4pt]
&=\frac{C\|\mathcal U\|_{L^{q_0}(B_1)}^{\frac{1}{4}\frac{\chi}{\chi-1}}}{(R-r)^{\frac{\chi}{\chi-1}}}\|\tilde{v}\|_{L^{1}(B_R)}^{\frac{q_0-1}{2q_0}}\cdot\|\tilde{v}\|_{L^{\infty}(B_R)}^{\frac{q_0+1}{2q_0}}\\[4pt]
&\leq \frac{1}{2}\|\tilde{v}\|_{L^{\infty}(B_R)}+\frac{C\|\mathcal U\|_{L^{q_0}(B_1)}^{\frac{\delta}{4}}}{(R-r)^{\delta}}\|\tilde{v}\|_{L^1(B_R)}.
\end{aligned}$$
where 
$\delta=\frac{\chi}{\chi-1}\cdot\frac{2q_0}{q_0-1}=\frac{q_0n}{q_0-n}$.
Set $f(t)=\|\tilde{v}\|_{L^\infty(B_t)}$ for $t\in (0,1].$ Then for any $0<r<R\leq 1$
$$f(r)\leq \frac{1}{2}f(R)+\frac{C\|\mathcal U\|_{L^{q_0}(B_1)}^{\frac{\delta}{4}}}{(R-r)^{\delta}}\|\tilde{v}\|_{L^1(B_1)}.$$
We apply  Lemma \ref{lem4.3} to get
$$f(r)\leq \frac{C\|\mathcal U\|_{L^{q_0}(B_1)}^{\frac{\delta}{4}}}{(R-r)^{\delta}}\|\tilde{v}\|_{L^1(B_1)}.$$
By choosing $r=\frac{1}{2}$, $R=1$, we have
\[\sup_{B_{1/2}}v\leq C\|\mathcal U\|_{L^{q_0}(B_1)}^{\frac{\delta}{4}}(\|v\|_{L^1(B_1)}+K)=C\|\mathcal U\|_{L^{q_0}(B_1)}^{\frac{\delta}{4}}(\|v\|_{L^1(B_1)}+\|f\|_{L^q(B_1)}).\]

The general case of  $R$, $r$, $p$ can be obtained by scaling and H\"older's  inequality.
For general $R>0$, define $\hat v(z)=v(Rz)$. It is clear that $\hat v$ satisfies
\[
\int_{B_1} \hat H_k^{i\bar j}D_{i}\hat v D_{\bar j}\varphi+\hat c\hat v\varphi\, d\mu_0\leq \int_{B_1} \hat f\varphi\,d\mu_0\quad \text{for\,\,any}\,\, \varphi\in W^{1,2}_0(B_1)\text{\ and}\,\,\varphi\geq 0
\]
where
\[\hat H_k^{i\bar j}(z)=H_k^{i\bar j}(Rz),\ \hat c(z)=R^2c(Rz), \ \hat f(z)=R^2f(Rz).\]
Applying the above estimate to $\hat v$ on $B_1$, we have
\[\sup_{B_{R/2}}v\leq C\lp\Theta_{q_0}(B_R) \rp^{\frac{\delta}{4}}\cdot (R^{-2n}\|v\|_{L^1(B_R)}+R^{2-\frac{2n}{q}}\|f\|_{L^q(B_R)}),\]
where $\Theta_{q_0}(B_R)=R^{-\frac{2n}{q_0}}\|\mathcal U\|_{L^{q_0}(B_R)}$.
Then by similar arguments as in \cite{HL} with H\"older's inequality, we finally obtain  \eqref{loc-b}
for general $R$, $r$, $p$.
\end{proof}

\begin{lem}[{\cite[Lemma 4.3]{HL}}]\label{lem4.3}
Let $f(t)\geq 0$ be bounded in $[\tau_0,\tau_1]$ with $\tau_0\geq 0$. Suppose for $\tau_0\leq t<s\leq \tau_1$ we have
$$f(t)\leq \beta f(s)+\frac{A}{(s-t)^\alpha}+B$$
for some $\beta\in[0,1)$. Then for any $\tau_0\leq t<s\leq\tau_1$ there holds
$$f(t)\leq C(\alpha,\beta)\left\{\frac{A}{(s-t)^\alpha}+B\right\}.$$
\end{lem}


\vskip 15pt

\subsection{Weak Harnack inequality}
In order to establish a weak Harnack inequality, the following Poincare type inequality is needed.
\begin{lem}\label{p-ineq}
Let $\Omega\subset\mathbb C^n$ be a bounded domain and $u\in C^2(\bar\Omega)$ a smooth, strictly $k$-plurisubharmonic function($2\leq k\leq n$) satisfying \eqref{low-upp}. 
Let $B_R\subset\Omega$ be a ball with radius $R$. Then for $1\leq p\leq \frac{2n}{2n-1}$,
\beq\label{poin-eq1}
\|v-v_R\|_{L^p(B_R)}\leq C\| \mathcal U\|_{L^{n-1}(B_R)}^\frac{n-1}{2}\cdot \lp\int_{B_R}H_k^{i\bar j}v_i\bar v_{j}\,\omega_0^n\rp^{\frac{1}{2}}, \ v\in W^{1,1}(B_R),
\eeq
where $v_R=\frac{1}{|B_R|}\int_{B_R}v\,\omega_0^n$ and $C$ depends on $k$, $n$, $R$, $p$, $\lambda$, $\Lambda$.
\end{lem}

\begin{proof}
By the usual Poincare inequality
\[\|v-v_R\|_{L^1(B_R)}\leq C\|Dv\|_{L^1(B_R)},\ v\in W^{1,1}(B_R).\]
Applying the usual Sobolev inequality to 
$v-v_R$, we have for $1\leq p\leq \frac{2n}{2n-1}$,
\[\|v-v_R\|_{L^p(B_R)}\leq C(\|Dv\|_{L^1(B_R)}+\|v-v_R\|_{L^1(B_R)})\leq C\|Dv\|_{L^1(B_R)}\]
for $v\in W^{1,1}(B_R)$.
By an elementary inequality
\[\|Dv\|_{L^1(B_R)}\leq \lp\int_{B_R}\sum_{i=1}^n\frac{1}{H_k^{i\bar i}}\,\omega_0^n\rp^{\frac{1}{2}} \lp\int_{B_R}H_k^{i\bar j}v_i\bar v_{j}\,\omega_0^n\rp^{\frac{1}{2}}.\]
Note that \cite{W} 
\[\prod_{i=1}^n\frac{\partial\sigma_k(\lambda)}{\partial \lambda_i}\geq C_{n,k}\sigma_k^{\frac{n(k-1)}{k}},\ \lambda\in\Gamma_k,\]
i.e., 
\[\frac{\partial\sigma_k(\lambda)}{\partial \lambda_i}\leq C^{-1}_{n,k}\sigma_k^{-\frac{n(k-1)}{k}} \prod_{j\neq i}^n\frac{\partial\sigma_k(\lambda)}{\partial \lambda_i}\leq C^{-1}_{n,k}\sigma_k^{-\frac{n(k-1)}{k}} \lp\frac{1}{n-1}\sum_{j\neq i}^n\frac{\partial\sigma_k(\lambda)}{\partial \lambda_i}\rp^{n-1}.\]
Then by \eqref{low-upp}, we know
$\frac{1}{H_k^{i\bar i}}\leq C\mathcal U^{n-1}$.
\end{proof}

\begin{rem}
Note that when $k=n$ and $u$ satisfies 
\beq\label{det-ul}
0<\lambda\leq \det(u_{i\bar j})\leq \Lambda,
\eeq
we have
$$\int_{B_R} \sum_i\frac{1}{U^{i\bar i}}\,\omega_0^n\leq \int_{B_R}C_n\triangle u\,\omega_0^n\leq \int_{\Omega}\triangle u\,\omega_0^n=\int_{\partial\Omega}\frac{\partial u}{\partial \gamma}\,d\sigma,$$
where $\gamma$ is the unit outer normal of $\partial\Omega$.
Hence, the constant in \eqref{poin-eq1} depends on $k$, $n$, $R$, $p$, $\lambda$, $\Lambda$ and $u|_{\partial\Omega}$.
\end{rem}

Now we are ready to prove the weak Harnack ineqaulity.

\begin{theo}[Weak Harnack inequality]\label{weak-har2}
Assume $u$ is a smooth, strictly $k$-plurisubharmonic function($2\leq k\leq n$) that satisfies \eqref{low-upp}. Assume $ f(z)\in L^q(B_1)$ with $q>n$.
Let $v\in W^{1,2}(B_1)$ be a nonnegative supersolution to 
\eqref{com-div-1}
in the following sense:
\begin{equation}\label{eq3}
	\int_{B_1} H_k^{i\bar j}D_{i}v D_{\bar j}\varphi\, d\mu_0\geq \int_{B_1} f\varphi\,d\mu_0\quad \text{for\,\,any}\,\, \varphi\in W^{1,2}_0(B_1)\text{\ and}\,\,\varphi\geq 0.
\end{equation}
Then any ball $B_R\subset B_1$ there holds for $0<p<\frac{n}{n-1}$, $n<q_0\leq q$ and $0<\theta<\tau<1$
\beq\label{w-har}
\inf_{B_{\theta R}}v+R^{2-\frac{2n}{q}}\|f\|_{L^{q}(B_R)}\geq C\cdot \Theta_{q_0}^{-\frac{\delta}{4}}(B_{\tau R})\cdot(\tau-\theta)^{\frac{2n}{p}}\left(\frac{1}{R^{2n}}\int_{B_{\tau R}}v^p\,d\mu_0\right)^{
	\frac{1}{p}},
\eeq
where $C$ depends only on $k$, $n$, $p$, $q_0$, $\lambda$, $\Lambda$, and $\|f\|_{L^q(B_1)}$.
\end{theo}

\begin{proof}
We prove for $R=1$. 

First, we prove the result for some $p_0>0$.
Set $\hat v=v+K>0$ for some $K>0$ to be determined and $\tilde v=\hat v^{-1}$. For $\varphi\in W^{1,2}_0(B_1)$ with $\varphi\geq 0$, we choose $\hat v^{-2}\varphi$ as the test function in \eqref{eq3}. We have $\tilde v$ satisfies
\[\int_{B_1} H_k^{i\bar j}D_{i}v \frac{D_{\bar j}\varphi}{\hat v^2}\, d\mu_0-2
\int_{B_1} H_k^{i\bar j}D_{i}v D_{\bar j}\hat v\frac{\varphi}{\hat v^3}\, d\mu_0
\geq \int_{B_1} f\frac{\varphi}{\hat v^2}\,d\mu_0,\]
i.e.,
\[\int_{B_1} H_k^{i\bar j}D_{i}\tilde v D_{\bar j}\varphi\, d\mu_0+ \int_{B_1} \tilde f\tilde v\varphi\,d\mu_0\leq 0,\]
where $\tilde f=\frac{f}{\hat v}$. We choose $K=\|f\|_{L^{q}(B_1)}$ such that $\|\tilde f\|_{L^{q}(B_1)}\leq 1$.
By Theorem \ref{thm-har}, for any $0< \theta<\tau<1$ and any $p>0$
\[\sup_{B_\theta} \hat v^{-1}=\sup_{B_\theta} \tilde v\leq \lp \tau^{-\frac{2n}{q_0}}\|\cU\|_{L^{q_0}(B_{\tau})}\rp^{\frac{\delta}{4}}\cdot \frac{1}{(\tau-\theta)^{\frac{2n}{p}}}\lp\int_{B_{\tau}}\hat v^{-p}\,d\mu_0\rp^{\frac{1}{p}},\]
where $\delta=\frac{nq_0}{q_0-n}$. This implies
\beqs
	\inf_{B_\theta}\hat v&\geq & C\cdot \Theta_{q_0}^{-\frac{\delta}{4}}(B_{\tau})(\tau-\theta)^{\frac{2n}{p}}\lp \int_{B_{\tau}}\hat v^{-p}\rp^{-\frac{1}{p}}\cdot\lp\int_{B_\tau} \hat v^{-p}\,d\mu_0\rp^{-\frac{1}{p}} \\
&= & C\cdot \Theta_{q_0}^{-\frac{\delta}{4}}(B_{\tau})(\tau-\theta)^{\frac{2n}{p}}\cdot\lp\int_{B_\tau} \hat v^{-p}\,d\mu_0\int_{B_\tau} \hat v^{p}\,d\mu_0\rp^{-\frac{1}{p}}\cdot\lp\int_{B_\tau} \hat v^{p}\,d\mu_0\rp^{\frac{1}{p}}.
\eeqs
Here $C$ depends on $p$, $q_0$, $k$, $n$, $\lambda$, $\Lambda$, $\|f\|_{L^q(B_1)}$.
In order to obtain \eqref{w-har}, it  suffices to prove that there exists $p_0>0$ such that
\beq\label{suff-ineq}
\int_{B_\tau} \hat v^{-p_0}\,d\mu_0\int_{B_\tau} \hat v^{p_0}\,d\mu_0\leq C.
\eeq
We will establish \eqref{suff-ineq} by showing that 
\beq\label{exp-ineq}
\int_{B_\tau} e^{p_0|w|}\,d\mu_0\leq C,
\eeq
where $w=\log\hat v-\gamma$ with $\gamma=\frac{1}{|B_\tau|}\int_{B_\tau} \log\hat v\,d\mu_0$.
By \eqref{exp-ineq}, we have
\[\int_{B_\tau} e^{p_0(\log\hat v-\gamma)}\,d\mu_0\leq C, \ \ \int_{B_\tau} e^{-p_0(\log\hat v-\gamma)}\,d\mu_0\leq C,
\] 
which implies \eqref{suff-ineq}.

We first derive the equation for $w$. 
For $\varphi\in W^{1,2}_0(B_1)$ with $\varphi\geq 0$, we choose $\hat v^{-1}\varphi$ as the test function in \eqref{eq3}. By similar calculation as above, we have
\beq\label{w-eq}
\int_{B_1} H_k^{i\bar j}D_{i}w D_{\bar j}w\varphi\, d\mu_0
\leq \int_{B_1} H_k^{i\bar j}D_{i}w D_{\bar j}\varphi\, d\mu_0+ \int_{B_1} -\tilde f\varphi\,d\mu_0.
\eeq

Replacing $\varphi$ by $\varphi^2$ in \eqref{w-eq}, we have
\[\int_{B_1} H_k^{i\bar j}D_{i}w D_{\bar j}w\varphi^2\, d\mu_0
\leq \int_{B_1} 2\varphi H_k^{i\bar j}D_{i}w D_{\bar j}\varphi\, d\mu_0+ \int_{B_1} -\tilde f\varphi^2\,d\mu_0.\]
By the mean value inequality
\[\int_{B_1} 2\varphi H_k^{i\bar j}D_{i}w D_{\bar j}\varphi\, d\mu_0\leq \frac{1}{2}\int_{B_1} H_k^{i\bar j}D_{i}w D_{\bar j}w\varphi^2\, d\mu_0+4\int_{B_1} H_k^{i\bar j}D_{i}\varphi D_{\bar j}\varphi\, d\mu_0.\]
By $q>n$, its conjugate $q^*\leq \frac{n}{n-1}$, i.e.,  $2q^*\leq \frac{2n}{n-1}$. Then Sobolev inequality  \eqref{sob-ineq},
\[ \int_{B_1} -\tilde f\varphi^2\,d\mu_0\leq \|\tilde f\|_{L^{q}(B_1)} \cdot \|\varphi\|_{L^{2q^*}(B_1)}^2\leq C \int_{B_1} H_k^{i\bar j}D_{i}\varphi D_{\bar j}\varphi\, d\mu_0.\]
Hence we have
\beq\label{L2-est}
\int_{B_1} H_k^{i\bar j}D_{i}w D_{\bar j}w\varphi^2\, d\mu_0\leq 
C \int_{B_1} H_k^{i\bar j}D_{i}\varphi D_{\bar j}\varphi\, d\mu_0.
\eeq
Now we take $\varphi\in C_0^1(B_1)$ with $\varphi=1$ in $B_\tau$.
Then we have
\[\int_{B_\tau} H_k^{i\bar j}D_{i}w D_{\bar j}w\, d\mu_0\leq 
C,\]
where $C$ depends on $k$, $n$, $\lambda$, $\Lambda$, $q$, $\tau$, $\|\mathcal U\|_{L^{1}(B_1)}$ and $\|f\|_{L^q(B_1)}$.
Note that $\frac{1}{|B_\tau|}\int_{B_\tau} w\,d\mu_0=0$.
By Poincare inequality \eqref{poin-eq1}, we obtain
\beq\label{ineq-2}
\int_{B_\tau} |w|^{\frac{2n}{2n-1}}\, d\mu_0\leq  C\lp\int_{B_1} H_k^{i\bar j}D_{i}w D_{\bar j}w\varphi^2\, d\mu_0\rp^{\frac{n}{2n-1}}\leq C,
\eeq 
$C$ depends on $k$, $n$, $q$, $\lambda$, $\Lambda$, $\tau$, $\|\mathcal U\|_{L^{q_0}(B_1)}$ and $\|f\|_{L^q(B_1)}$.
Furthermore, by \eqref{L2-est},
we have for $\tau'\in (\tau, 1)$
\beq\label{ineq-21}
\int_{B_{\tau'}} |w|^{\frac{2n}{2n-1}}\, d\mu_0\leq C,
\eeq 
$C$ depends on $n$, $q$, $\lambda$, $\Lambda$, $\tau'$, $\|\mathcal U\|_{L^{q_0}(B_1)}$ and $\|f\|_{L^q(B_1)}$.

Next we estimate $\int_{B_\tau} w^{\beta}\, d\mu_0$ for any $\beta\geq 2$ by iteration.
For $m>0$, choose $\varphi=\zeta^2|w|^{2\beta}$ for a cut-off function $\zeta$ in $B_1$,
where 
\[w_m=
\begin{cases}
	-m,&\  w\leq -m,\\
	w, &  \ |w|<m, \\
	m,& \ w\geq m.
\end{cases}\]
Note that $Dw=Dw_m$ for $|w|<m$ and $D_wm=0$ for $|w|>m$.
By \eqref{w-eq}, we have
\begin{eqnarray*}
	\int_{B_1} H_k^{i\bar j}D_{i}w D_{\bar j}w\zeta^2|w_m|^{2\beta}\, d\mu_0
	&\leq& 2\beta\int_{B_1} \zeta^2H_k^{i\bar j}D_{i}w D_{\bar j}|w_m||w_m|^{2\beta-1}\, d\mu_0\\
	&&+\int_{B_1} 2\zeta |w_m|^{2\beta} H_k^{i\bar j}D_{i}w D_{\bar j}\zeta\, d\mu_0+ \int_{B_1} |\tilde f| \zeta^2|w_m|^{2\beta}\,d\mu_0\\
	&\leq& 2\beta\int_{B_1} \zeta^2H_k^{i\bar j}D_{i}w_m D_{\bar j}w_m|w_m|^{2\beta-1}\, d\mu_0\\
	&&+\int_{B_1} 2\zeta |w_m|^{2\beta} H_k^{i\bar j}D_{i}w D_{\bar j}\zeta\, d\mu_0+ \int_{B_1} |\tilde f| \zeta^2|w_m|^{2\beta}\,d\mu_0\\
	&\leq& 
	(1-\frac{1}{2\beta})\int_{B_1} \zeta^2 |w_m|^{2\beta}H_k^{i\bar j}D_{i}w_m D_{\bar j}w_m\, d\mu_0\\
	&&+(2\beta)^{2\beta-1}\int_{B_1} \zeta^2 H_k^{i\bar j}D_{i}w_m D_{\bar j}w_m\, d\mu_0\\
	&&+\int_{B_1} 2\zeta |w_m|^{2\beta} H_k^{i\bar j}D_{i}w D_{\bar j}\zeta\, d\mu_0+ \int_{B_1} |\tilde f| \zeta^2|w_m|^{2\beta}\,d\mu_0.
\end{eqnarray*}
Here we used  Young's inequality 
\[2\beta |w_m|^{2\beta-1}\leq (1-\frac{1}{2\beta})|w_m|^{2\beta}+(2\beta)^{2\beta-1}.\]
Therefore we have
\begin{eqnarray*}
	\int_{B_1} H_k^{i\bar j}D_{i}w D_{\bar j}w\zeta^2|w_m|^{2\beta}\, d\mu_0
	&\leq& 
	(2\beta)^{2\beta}\int_{B_1} \zeta^2 H_k^{i\bar j}D_{i}w_m D_{\bar j}w_m\, d\mu_0\\
	&&+4\beta\int_{B_1} \zeta |w_m|^{2\beta} H_k^{i\bar j}D_{i}w D_{\bar j}\zeta\, d\mu_0+ 2\beta\int_{B_1} |\tilde f| \zeta^2|w_m|^{2\beta}\,d\mu_0.
\end{eqnarray*}
By the mean value inequality
\[4\beta\int_{B_1} \zeta |w_m|^{2\beta} H_k^{i\bar j}D_{i}w D_{\bar j}\zeta\, d\mu_0\leq
\int_{B_1} (\frac{1}{2}H_k^{i\bar j}D_{i}w D_{\bar j}w\zeta^2+C\beta^2H_k^{i\bar j}D_{i}\zeta D_{\bar j}\zeta)|w_m|^{2\beta}\, d\mu_0.\]
Then we obtain
\begin{eqnarray*}
	\int_{B_1} H_k^{i\bar j}D_{i}w_m D_{\bar j}w_m\zeta^2|w_m|^{2\beta}\, d\mu_0
	&\leq& 
	\left[\lp(2\beta)^{2\beta}+\frac{1}{2}\rp\int_{B_1} \zeta^2 H_k^{i\bar j}D_{i}w_m D_{\bar j}w_m\, d\mu_0
	\right.\\ &&\left.+C\beta^2\int_{B_1} H_k^{i\bar j}\zeta_{i} \zeta_{\bar j}|w_m|^{2\beta}\, d\mu_0+ 2\beta\int_{B_1} |\tilde f| \zeta^2|w_m|^{2\beta}\,d\mu_0\right].
\end{eqnarray*}
For convenience, we write $w=w_m$  and then let $m\to+\infty$.
By Young's inequality
\begin{eqnarray*}
	H_k^{i\bar j}D_{i}(\zeta |w|^\beta) D_{\bar j}(\zeta |w|^\beta)&=&\zeta^2\beta^2|w|^{2\beta-2}H_k^{i\bar j}D_{i}w D_{\bar j}w+|w|^{2\beta}H_k^{i\bar j}D_{i}\zeta D_{\bar j}\zeta\\
	&&+\zeta\beta|w|^{2\beta-1}H_k^{i\bar j}D_{i}w D_{\bar j}\zeta+\zeta\beta|w|^{2\beta-1}H_k^{i\bar j}D_{i}\zeta D_{\bar j}w\\
	&\leq&2\zeta^2\beta^2|w|^{2\beta-2}H_k^{i\bar j}D_{i}w D_{\bar j}w+2|w|^{2\beta}H_k^{i\bar j}D_{i}\zeta D_{\bar j}\zeta\\
	&\leq&2\zeta^2(\frac{\beta-1}{\beta}|w|^{2\beta}+\beta^{2\beta-1})H_k^{i\bar j}D_{i}w D_{\bar j}w+2|w|^{2\beta}H_k^{i\bar j}D_{i}\zeta D_{\bar j}\zeta.
\end{eqnarray*}
Hence we have
\begin{eqnarray*}
	&&\int_{B_1} H_k^{i\bar j}D_{i}(\zeta |w|^\beta) D_{\bar j}(\zeta |w|^\beta)\, d\mu_0\\
	&\leq& 
	(2\beta)^{2\beta}C\int_{B_1} \zeta^2 H_k^{i\bar j}D_{i}w D_{\bar j}w\, d\mu_0
	+C\beta^2\int_{B_1} H_k^{i\bar j}D_{i}\zeta D_{\bar j}\zeta|w|^{2\beta}\, d\mu_0+ 4\beta\int_{B_1} |\tilde f| \zeta^2|w|^{2\beta}\,d\mu_0,
\end{eqnarray*}
where $C$ relies only on $k$, $n$. 
By H\"older's inequality
\[\int_{B_1} |\tilde f| \zeta^2|w|^{2\beta}\,d\mu_0\leq \lp\int_{B_1} |\tilde f|^q\,d\mu_0\rp^{\frac{1}{q}}\left[\int_{B_1} (\zeta |w|^{\beta})^{\frac{2q}{q-1}}\,d\mu_0\right]^{\frac{q-1}{q}}.\]
Since $q>n$, $\frac{2q}{q-1}<\frac{2n}{n-1}$. 
Then we apply Sobolev inequality \eqref{sob-ineq} to get
\begin{eqnarray*}
	\|\zeta |w|^{\beta}\|_{L^{\frac{2q}{q-1}}(B_1)}&\leq& \epsilon\|\zeta |w|^{\beta}\|_{L^{\frac{2n}{n-1}}(B_1)}+C_{n, q}\eps^{-\frac{n}{2q-n}}\|\zeta |w|^{\beta}\|_{L^{2}(B_1)}\\
	&\leq& \epsilon \left[\int_{B_1} H_k^{i\bar j}D_{i}(\zeta |w|^\beta) D_{\bar j}(\zeta |w|^\beta)\, d\mu_0\right]^{\frac{1}{2}}+C_{n, q}\eps^{-\frac{n}{2q-n}}\|\zeta |w|^{\beta}\|_{L^{2}(B_1)}.
\end{eqnarray*}
Here $\epsilon$ is a small constant.
Therefore by \eqref{L2-est} we have
\begin{eqnarray*}
	&&\int_{B_1} H_k^{i\bar j}D_{i}(\zeta |w|^\beta) D_{\bar j}(\zeta |w|^\beta)\, d\mu_0\\
	&\leq& 
	C\left[(2\beta)^{2\beta}\int_{B_1} \zeta^2 H_k^{i\bar j}D_{i}w D_{\bar j}w\, d\mu_0
	+\beta^\alpha\int_{B_1} (H_k^{i\bar j}D_{i}\zeta D_{\bar j}\zeta +\zeta^2)|w|^{2\beta}\, d\mu_0\right]\\
	&\leq& 
	C\left[(2\beta)^{2\beta}\int_{B_1} H_k^{i\bar j}D_{i}\zeta D_{\bar j}\zeta\, d\mu_0
	+\beta^\alpha\int_{B_1} (H_k^{i\bar j}D_{i}\zeta D_{\bar j}\zeta +\zeta^2)|w|^{2\beta}\, d\mu_0\right].
\end{eqnarray*}
where $\alpha>0$ is a constant depending on $n$ and $q$.
Now by  Sobolev inequality \eqref{sob-ineq} for $\zeta |w|^\beta$ with $\chi_0=\frac{n}{n-1}$,
we obtain
\begin{eqnarray*}
	\lp \int_{B_1}\zeta^{2\chi_0}|w|^{2\beta\chi_0}\rp^{\frac{1}{\chi_0}}\, d\mu_0&\leq& 
	C\beta^\alpha\left[(2\beta)^{2\beta}\int_{B_1} H_k^{i\bar j}\zeta_i \zeta_{\bar j}\, d\mu_0
	+\int_{B_1} (H_k^{i\bar j}\zeta_i \zeta_{\bar j} +\zeta^2)|w|^{2\beta}\, d\mu_0\right].
\end{eqnarray*}

For $\tau\leq r<R<1$, 
we choose the cut-off function such that $\zeta=1$  on $B_r$, $\zeta=0$ in $B_1\setminus B_R$, 
and $|D\zeta|\leq \frac{2}{R-r}$. Note that $\mathcal U\in L^{q_0}$. Then 
\begin{eqnarray*}
	\int_{B_1} (H_k^{i\bar j}D_{i}\zeta D_{\bar j}\zeta +\zeta^2)|w|^{2\beta}\, d\mu_0&\leq&
	\frac{C}{(R-r)^2}\int_{B_R} (\mathcal U+C)|w|^{2\beta}\, d\mu_0\\
	&\leq&
	\frac{C\|\mathcal U\|_{L^{q_0}(B_R)}}{(R-r)^2}\lp\int_{B_R} |w|^{2\beta \frac{q_0}{q_0-1}}\, d\mu_0\rp^{\frac{q_0-1}{q_0}}.
\end{eqnarray*}
Denote $\chi=\frac{(q_0-1)\chi_0}{q_0}>1$. Hence we get
\begin{eqnarray*}
	\lp \int_{B_r}|w|^{\frac{2\beta q_0}{q_0-1}\chi }\, d\mu_0\rp^{\frac{1}{\chi_0}}&\leq& 
	\frac{C\beta^\alpha \|\cU\|_{L^{q_0}(B_R)}}{(R-r)^2} \left[(2\beta)^{2\beta}
	+\lp\int_{B_R} |w|^{ \frac{2\beta q_0}{q_0-1}}\, d\mu_0\rp^{\frac{q_0-1}{q_0}} \right].
\end{eqnarray*}
For $\tau' \in (\tau, 1)$, we set $\beta_i=\chi^{i-1}$ and $r_i=\tau+\frac{\tau'-\tau}{2^{i-1}}\le \tau'$
for any $i=1, 2, \cdots$. Then
\begin{eqnarray*}
	\lp \int_{B_{r_i}}|w|^{\frac{2 q_0}{q_0-1}\chi^i }\, d\mu_0\rp^{\frac{1}{\chi_0}}&\leq& 
	\frac{C\chi^{\alpha(i-1)}2^{2(i-1)}\|\cU\|_{L^{q_0}(B_{\tau'})}}{(\tau'-\tau)^2}\\
	&& \cdot\left[(2\chi^{i-1})^{2\chi^{i-1}}
	+\lp\int_{B_{r_{i-1}}} |w|^{ \frac{2 q_0}{q_0-1}\chi^{i-1}}\, d\mu_0\rp^{\frac{q_0-1}{q_0}} \right].
\end{eqnarray*}
Let $$I_j=\|w\|_{L^{\frac{2 q_0}{q_0-1}\chi^j}(B_{r_j})},\ \ \ j=1,2,\cdots.$$
Then
\[I_j\leq \lp\frac{C\chi^{\alpha(j-1)}2^{2(j-1)}\|\cU\|_{L^{q_0}(B_{\tau'})}^{j-1}}{(\tau'-\tau)^2}\rp^{\frac{1}{2\chi^{j-1}}}(2\chi^{j-1}+I_{j-1})\leq
\frac{\lp C\|\cU\|_{L^{q_0}(B_{\tau'})}\rp^{\frac{j-1}{\chi^j-1}}}{(\tau'-\tau)^{\frac{1}{\chi^{j-1}}}}(2\chi^{j-1}+I_{j-1})\]
with $C$ depending on $k$, $n$, $\lambda$, $\Lambda$ and $\|f\|_{L^q(B_1)}$.
Iterating this inequality with 
\[\sum_{i=0}^\infty \frac{i}{\chi^i}=\sum_{N=1}^{\infty}\sum_{i=N}^{\infty}\frac{1}{\chi^i}=\frac{\chi}{(\chi-1)^2},\] 
we have
\begin{align*}
I_j\le \frac{C\|\cU\|_{L^{q_0}(B_{\tau'})}^{\sum_{i=1}^j\frac{i-1}{\chi^{i-1}}}}{(\tau'-\tau)^{\sum_{i=1}^j\frac{1}{\chi^{i-1}}}}\lp \chi^{j-1}+I_0\rp.
\end{align*}

Now for any $\beta$, there exists a $j$ such that  $\frac{2 q_0}{q_0-1}\chi^{j-1}< \beta\leq\frac{2 q_0}{q_0-1}\chi^j$. Hence, 
\begin{eqnarray*}
	\|w\|_{L^{\beta}(B_{\tau})}=\lp\int_{B_\tau}|w|^\beta\rp^{\frac{1}{\beta}}\leq CI_j
	\leq \frac{C\|\cU\|_{L^{q_0}}^{\frac{\chi}{(\chi-1)^2}}}{(\tau'-\tau)^{\frac{1}{\chi-1}}}(\beta+I_0).
\end{eqnarray*}
By \eqref{ineq-21}, $\|w\|_{L^{\frac{2 n}{2n-1}}(B_{\tau'})}\leq C$. 
Hence, we have
\begin{eqnarray*}
	\frac{C\|\cU\|_{L^{q_0}(B_{\tau'})}^{\frac{\chi}{(\chi-1)^2}}}{(\tau'-\tau)^{\frac{1}{\chi-1}}}I_0
	&=&\frac{C\|\cU\|_{L^{q_0}(B_{\tau'})}^{\frac{\chi}{(\chi-1)^2}}}{(\tau'-\tau)^{\frac{1}{\chi-1}}}\|w\|_{L^{\frac{2 q_0}{q_0-1}}(B_{\tau'})}\\[5pt]
	&\leq& \frac{C\|\cU\|_{L^{q_0}(B_{\tau'})}^{\frac{\chi}{(\chi-1)^2}}}{(\tau'-\tau)^{\frac{1}{\chi-1}}}\|w\|_{L^{\frac{2 n}{2n-1}}(B_{\tau'})}^{\kappa}\cdot \|w\|_{L^{\beta}(B_{\tau'})}^{1-\kappa}\\[5pt]
	&\leq & \frac{1}{2}\|w\|_{L^{\beta}(B_{\tau'})}+\frac{C\|\cU\|_{L^{q_0}(B_{\tau'})}^{\frac{\chi}{(\chi-1)^2}}}{(\tau'-\tau)^{\frac{C}{\chi-1}}},
\end{eqnarray*}
where 
\[\kappa=\frac{\frac{q_0-1}{2 q_0}-\frac{1}{\beta}}{\frac{2n-1}{2 n}-\frac{1}{\beta}}.\]
This implies
\[\|w\|_{L^{\beta}(B_{\tau})}\leq  \frac{1}{2}\|w\|_{L^{\beta}(B_{\tau'})}+\frac{C\|\cU\|_{L^{q_0}(B_{\tau'})}^{\frac{2\chi}{(\chi-1)^2}}\beta}{(\tau'-\tau)^{\frac{2}{\chi-1}}},\]
with $C$ depending on $k$, $n$, $\lambda$, $\Lambda$ and $\|f\|_{L^q(B_1)}$.
By  Lemma \ref{lem4.3}, 
\[\|w\|_{L^{\beta}(B_{\tau})}\leq \frac{C\|\cU\|_{L^{q_0}}^{\frac{2\chi}{(\chi-1)^2}}\beta}{(\tau'-\tau)^{\frac{2}{\chi-1}}}\leq C_*\beta,\]
where 
$$C_*:=\frac{C\|\cU\|_{L^{q_0}(B_{\tau'})}^{\frac{2\chi}{\chi-1}}}{(\tau'-\tau)^{\frac{2}{\chi-1}}},$$
and $C$ depends on $k$, $n$, $\lam$, $\Lambda$ and $\|f\|_{L^q(B_R)}$. 
Therefore, for any integer $\beta\geq 1$,
\[\int_{B_\tau}|w|^\beta\,d\mu_0\leq C_*^\beta  \beta^\beta\leq C_*^\beta e^\beta\beta!.\]
Here we used the Sterling formula. Let $p_0=\frac{1}{2C_*e}$. We have
\[\int_{B_\tau}\frac{(p_0|w|)^\beta}{\beta!}\,d\mu_0\leq \frac{1}{2^\beta}.\]
By Taylors series, we have 
\[\int_{B_\tau}e^{p_0|w|}\,d\mu_0\leq 2.\]
In conclusion, we have proved \eqref{w-har} for $p=p_0$ and $R=1$. 

As in \cite[Theorem 4.15, Step 2]{HL},
we can prove
\[\lp\int_{B_{r_1}}\hat v^{p_1}\,d\mu_0\rp^{\frac{1}{p_1}}\leq \lp\int_{B_{r_1}}\hat v^{p_2}\,d\mu_0\rp^{\frac{1}{p_2}}\] 
for $0<r_1<r_2<1$ and $0<p_2<p_1\le\frac{2n}{n-1}$ by similar arguments as in Theorem \ref{thm-har}. 
Then by iterations, we get \eqref{w-har} for $p\le\frac{2n}{n-1}$ and $R=1$. 
We omit the details here. Finally, by 
scaling as in Theorem \ref{thm-har}, we complete the proof of  \eqref{w-har}.
\end{proof}

\vskip 15pt

Combining Theorem \ref{thm-har} and Theorem \ref{weak-har2},  
we have  Theorem \ref{thm-har3}  and Theorem \ref{int-holder}  by standard arguments \cite{GT, HL}.

\vskip 20pt

\section{Remarks of the real Hessian case}\label{sec-hes}

The results in previous sections can be extended to real Hessian equations.

Let  $\Omega\in \mathbb R^n$ be a bounded domain and $u$ a smooth function defined on $\Omega$.
The $k$-complex Hessian operator is defined by 
\beq\label{k-r-hess}
S_k[u]=\sum_{1\leq j_1<\cdots<\lambda_k\leq n}\lambda_{j_1}\cdots\lambda_{j_k}, \ k=1, \cdots, n,
\eeq
where $\lambda_1,\cdots,\lambda_n$ are the eigenvalues of the Hessian $(u_{ij})=\left(\frac{\partial^2u}{\partial x_i x_j}\right)$.
Denote $S^{i j}_k=\frac{\partial S_k[u]}{\partial u_{i j}}$. The {\it linearized $k$-Hessian equation}
is 
\beq\label{l-k-r-hess}
L_{u,k}v:=S^{i j}_k v_{i j}=f(x).
\eeq
In particular, when $k=n$, $S_n$ is the Monge-Amp\`ere operator and $L_{u,n}$ is the linearized Monge-Amp\`ere equation. Denote
$\mathcal S(x)=\sum_{i=1}^n S_k^{ii}$ and 
$$\theta_q(B_R(x))=\lp\frac{1}{|B_R(x)|}\int_{B_R(x)}\mathcal S^q\,d\mu_0\rp^{\frac{1}{q}}.$$

Since the proofs are almost the same as the complex case, we only list the main theorems without proofs. For the Sobolev inequalities, one can see that the only difference is the critical exponent.

\begin{theo}\label{Sobolev-r-hess}
Assume $u$ is a smooth, strictly $k$-convex function($2\leq k\leq n-1$) satisfying 
\beq\label{low-upp-r}
0<\lambda\leq S_k[u]\leq \Lambda.
\eeq  
Then for any 
$2\le p\leq \frac{2n}{n-2}$, there exists a constant $C$ depending on $k$, $n$, $\lambda$, $\Lambda$, $\Omega$ and $p$ such that for any $v\in C^{2}(\Om)$ with vanishing boundary value, there holds
\beq\label{sob-ineq-r}
\|v\|_{L^p(\Omega)}\le C\cdot \lp \int_{\Om}S_k^{i j}v_i v_{j}\,d\mu_0\rp^{\frac{1}{2}}.
\eeq
\end{theo}

In \cite{U90}, Urbas conjectured that  a $W^{2,p}$ solution to the $k$-Hessian equation with smooth right hand side is smooth when $p>\frac{k(k-1)}{2}$. This is known to be true when $k=n$, i.e., for the Monge-Amp\`ere equation. When $k<n$, 
it is proved that the conclusion holds if $p>\frac{k(n-1)}{2}$ \cite{U01}. The exponent was   improved to $p>\frac{n(k-1)}{2}$
as well as for the complex case \cite{DK}. By applying the Sobolev inequality \eqref{sob-ineq-r}, we have
\begin{theo}\label{main2-r}
Assume $n\geq 2$ and $p\geq \frac{n(k-1)}{2}$. Suppose $\Omega$ is a domain in $\mathbb R^n$. 
Suppose $u\in W^{2,p}(\Omega)$ be a solution to \eqref{k-hess-c}. Then for $\Omega'\Subset\Omega$, we have 
\begin{align}
\sup_{\Omega'}|\triangle u|\le C,
\end{align}
where $C$ depends on 
$k$, $n$, $p$, $\inf_\Omega f$, $\sup_\Omega f$, $\|f\|_{C^1(\Omega)}$, $\|u\|_{W^{2,p}(\Omega)}$ and $dist(\Omega', \p\Omega)$.
\end{theo}

We also have analogous regularity results for the linearized $k$-Hessian equations.

\begin{theo}\label{thm-har3-r}
Assume $u$ is a smooth, strictly $k$-convex function($2\leq k\leq n-1$) that satisfies 
\eqref{low-upp-r}.
Assume $ f\in L^q(\Omega)$ with $q>\frac{n}{2}$.
Let $v\in W^{1,2}(B_{R_0}(x_0))$ be a nonnegative solution to \eqref{l-k-r-hess} in $B_{R_0(x_0)}\subset \Omega$. Then there holds for any $q_0>\frac{n}{2}$ and $R<R_0$
\beq\label{cma-har}
\max_{B_{R}(x_0)} v\leq C(\min_{B_{R_0}(x_0)} v+R^{2-\frac{n}{q}}\|f\|_{L^{q}(B_{R_0}(x_0))}).
\eeq
Here $C$ depends only on $k$, $n$, $p$, $q_0$, $q$, $\lambda$, $\Lambda$,  $\|f\|_{L^q(\Omega)}$ and $\theta_{q_0}(B_{R_0}(x_0))$.
\end{theo}

\begin{cor}\label{int-holder-r}
Assume $u$ is a smooth, strictly $k$-convex function($2\leq k\leq n-1$) that satisfies 
\eqref{low-upp-r}. Assume $ f\in L^q(\Omega)$ with $q>\frac{n}{2}$.
Let $v\in W^{1,2}(\Omega)$ be a solution to \eqref{l-k-r-hess}.  
$\theta_{q_0}(B_R(x))$ is uniformly bounded for $B_R(x)\subset\Omega$ for some $q_0>\frac{n}{2}$.
Then $v\in C^\alpha(\Omega)$ for some 
$\alpha\in(0,1)$ depending on $k$, $n$, $q_0$, $q$, $\lambda$, $\Lambda$,  $\|f\|_{L^q(\Omega)}$ and $\theta_{q_0}=\sup_{B_R(x)}\theta_{q_0}(B_R(x))$. Moreover, there holds for any $B_R\subset\Omega$
\[|v(x)-v(y)|\leq C\lp\frac{|x-y|}{R}\rp^\alpha(\|v\|_{L^2(B_R)} +R^{2-\frac{n}{q}}\|f\|_{L^{q}(B_R)}), \ \ x, y\in B_{R/2},\]
 where the constant $C$ depends on $k$, $n$, $q_0$, $q$, $\lambda$, $\Lambda$,  $\|f\|_{L^q(\Omega)}$ and $\theta_{q_0}$.
\end{cor}

\vskip 20pt

\section{Applications to the cscK equation}\label{sec-scalar}

In this section, we consider the following fourth order equation
\beq\label{scalar-eq}
\sum_{i,j=1}^nu^{i\bar j}[\log\det(u_{kl})]_{i\bar j}=f(z),
\eeq
where $u\in C^4$ is a plurisubharmonic function. 
The left hand side of the equation is the scalar curvature of the K\"ahler metric given by $\omega=\sqrt{-1}\partial\bar\partial u$. Denote $\mathcal U=\sum_{i=1}^n\frac{1}{u_{i\bar i}}$. 
The following interior estimate was obtained in \cite{CC}.

\begin{theo} [{\cite[Corollary 6.2]{CC}}]
Let $u\in C^4(B_1)\cap PSH(B_1)$ be a bounded solution to the scalar flat
equation
\beq\label{scalar-flat}
\sum_{i,j=1}^nu^{i\bar j}[\log\det(u_{k\bar l})]_{i\bar j}=0.
\eeq
Assume that for some $p>3n(n-1)$, we have $\triangle u\in L^p(B_1)$ and $\mathcal U\in  L^p(B_1)$. Then for any $\alpha\in (0,1)$, 
\[\|u\|_{C^{2,\alpha}(B_{1/2})}\leq C.\]
Here $C$ depends on $p$, $\|u\|_{L^\infty(B_1)}$, $\|\triangle u\|_{L^p(B_1)}$, $\|\mathcal U\|_{L^p(B_1)}$.
\end{theo}

When $u=0$ on $\partial\Omega$, $\Omega$ is close to a ball and $\det(u_{i\bar j})$ is sufficiently close to a constant, 
it is show in \cite{CX} that  the conditions $\|\triangle u\|_{L^p(B_1)}$, $\|\mathcal U\|_{L^p(B_1)}$ automatically satisfy.
In this section, we study the interior regularity of \eqref{scalar-eq} with the new Sobolev inequality under 
the condition   
\beq\label{det-low-upp}
0<\lambda\leq \det(u_{i\bar j})\leq \Lambda.
\eeq

\begin{prop} 
Let $u\in C^4(B_1)\cap PSH(B_1)$ be a bounded solution to the scalar curvature
equation \eqref{scalar-eq}. 
Assume that \eqref{det-low-upp} holds and for some $p>n$, we have $\triangle u\in L^p(B_1)$, $\mathcal U\in  L^p(B_1)$ and $f\in W^{1, p}(B_1)$. Then 
\[\|u\|_{C^{2}(B_{1/2})}\leq C.\]
Here $C$ depends on $p$, $\lambda$, $\Lambda$, $\|u\|_{L^\infty(B_1)}$ and $\|\mathcal U\|_{L^p(B_1)}$.
\end{prop}

\begin{proof}
Let $G=\log\det(u_{k\bar l})$. Then \eqref{scalar-eq} can be written as
\beq\label{4-eq-g}
\triangle_\omega G= f(z),
\eeq
where $\triangle_\omega$ is the Laplacian with respect to $\omega$. 
Following the idea of \cite{CC}, we consider 
\[v=\triangle u+e^{\frac{1}{2}G}|\nabla_\omega G|^2,\]
where $|\nabla_\omega G|^2=u^{i\bar j}G_iG_{\bar j}$ is the gradient with respect to $\omega$. By checking the computations in \cite[Section 6]{CC},
$v$ satisfies the following differential inequality
\beq\label{diff-ineq}
\triangle_\omega v\geq -C(|f|+|\nabla f|+e^{-\frac{1}{2}G}\triangle u ) v.
\eeq
By the conditions in this proposition, we have $(|f|+|\nabla f|+e^{-\frac{1}{2}G}\triangle u )\in L^p(B_1)$
with $p>n$. Then we can apply Theorem \ref{thm-har} to obtain $\sup_{B_{1/2}} v\leq C\|v\|_{L^1(B_{3/4})}$.
It remains to prove $\|v\|_{L^1(B_{3/4})}\leq C$. 
It is clear that $\Delta u\in L^1(B_{1})$. It suffices to estimate the integral of $|\nabla_\omega G|^2$.
Let $\eta\in C_0^\infty(B_1)$  be a cutoff function and $\eta\equiv 1$ in ${B_{3/4}}$.
Multiplying \eqref{4-eq-g} by  $\eta^2G$ and by integration by parts, we have
$$\int_{B_1}U^{i\bar j}G_iG_{\bar j}\eta^2\,\omega_0^n+2\int_{B_1}U^{i\bar j}G_i\eta_{\bar j}\eta G\,\omega_0^n=\int_{B_1}fG\eta^2\,\omega_0^n.$$
Then by Cauchy's inequality, we get
\begin{eqnarray*}
	\frac{1}{2}\int_{B_1}U^{i\bar j}G_iG_{\bar j}\eta^2\,\omega_0^n&\leq& 2\int_{B_1}U^{i\bar j}\eta_i\eta_{\bar j} G^2\,\omega_0^n+\int_{B_1}|f|G\eta^2\,\omega_0^n\\[5pt]
	&\leq&2\int_{B_1}\mathcal U|D\eta|^2G^2\,\omega_0^n+\int_{B_1}|f|G\eta^2\,\omega_0^n.
\end{eqnarray*}
Hence $\int_{B_{3/4}}U^{k\bar l}G_kG_{\bar l}\,\omega_0^n\leq C$ follows by $\mathcal U\in L^{p}(B_1)$$(p>n)$, $f\in L^q(B_1)$.
\end{proof}

The higher regularities will follow by the standard  theory of elliptic equations.
Finally, by a scaling argument as in \cite{CC}, we have the Liouville theorem. 

\begin{cor} 
Let $u$ be a smooth, strictly plurisubharmonic function such that $\omega=\sqrt{-1}\partial\bar\partial u$ defines a scalar flat metric on $\mathbb C^n$, i.e., $u$ satisfies equation \eqref{scalar-flat}. Suppose \eqref{det-low-upp} holds on $\mathbb C^n$.
If for some $p>n$, we have 
\[\liminf_{r\to+\infty}\frac{1}{R^{2n}}\int_{B_R(0)\subset\mathbb C^n}(\triangle u+\mathcal U)^p \,d\mu_0<+\infty,\]
then 
$u$ is a quadratic polynomial.
\end{cor}


\vskip 20pt


\begin{thebibliography}{XXX}

\bibitem[BB]{BB}
R. Berman and R. Berndtsson, 
Moser-Trudinger type inequalities for complex Monge-Amp\`ere operators and Aubin’s ``hypoth\`ese fondamentale".  
{\it Ann. Fac. Sci. Toulouse Math.} {\bf 31}(2022), 595-645. 

\bibitem[BD]{BD}
Z. Blocki and S. Dinew, 
A local regularity of the complex Monge-Amp\`re equation. 
{\it Math. Ann.} {\bf 351}(2011), 411-416. 

\bibitem[BS]{BS}
P. Bella and M. Sch\"affner,
Local boundedness and Harnack inequality for solutions of linear nonuniformly elliptic equations.
{\it Comm. Pure Appl. Math.} {\bf 74}(2021), no.3, 453-477.

\bibitem [CC]{CC}
X. X. Chen and J. R. Cheng,
On the constant scalar curvature K\"ahler metrics, apriori estimates. 
\textit{arXiv:1712.06697}.

\bibitem[CX]{CX}
J. R.  Cheng and Y. L. Xu,
Interior $W^{2,p}$ estimate for small perturbations to the complex Monge-Amp\`ere equation.
\textit{arXiv: 2301.0094}.

\bibitem [CG]{CG}
L. Caffarelli and C. Guti\'errez, 
Properties of the solutions of the linearized Monge-Amp\`ere equation.
{\it Amer. J. Math.} {\bf 119}(1997), 423-465.

\bibitem[CNS]{CNS}
L. Caffarelli, L. Nirenberg and J. Spruck, 
Dirichlet problem for nonlinear second order elliptic equations III, Functions of the eigenvalues of the Hessian. 
{\it Acta Math.} {\bf 155}(1985), 261-301.

\bibitem[D1]{D1}
S. Donaldson,
Interior estimates for solutions of Abreu's equation.
{\it Collect. Math.}  {\bf 56}(2005),  no. 2, 103-142.

\bibitem[D2]{D2} 
S. Donaldson,
Extremal metrics on toric surfaces: a continuity method. 
{\it J. Diff. Geom. } {\bf 79}(2008),  no. 3, 389-432.

\bibitem[D3]{D3} 
S. Donaldson,
Constant scalar curvature metrics on toric surfaces.  
{\it Geom. Funct. Anal.}  {\bf 19}(2009),  no. 1, 83-136.

\bibitem[DFS]{DFS}
G. De Philippis, A. Figalli and O. Savin,
A note on interior $ W^{2, 1+\varepsilon} $ estimates for the Monge-Amp\`ere equation.
{\it Math. Ann.} {\bf 357}(2013), 11-22.

\bibitem[DK]{DK}
S. Dinew and S. Ko\l odziej,
A priori estimates for complex Hessian equations.
{\it Anal. PDE} {\bf 7}(2014): 227-244.

\bibitem[FSS]{FSS}
B. Franchi, R. Serapioni and F. Serra Cassano, 
Irregular solutions of linear degenerate elliptic equations. 
{\it Potential Anal.} {\bf 9} (1998), no. 3, 201-216.

\bibitem[GPS]{GPS}
B. Guo, D. H. Phong and J. Sturm,
Green's functions and complex Monge-Amp\`ere equations.
{\it arXiv:2202.04715}, to appear in {\it Jour. Diff. Geom}.



\bibitem[GT]{GT}
D. Gilbarg and N. Trudinger,
Elliptic partial differential equations of second order.
Reprint of the 1998 edition. {\it Classics in Mathematics}. 
Springer-Verlag, Berlin, 2001.

\bibitem[GW]{GW}
M. Gr\"{u}ter and K.-O. Widman,
The Green function for uniformly elliptic equation.
{\it Manuscripta Math.} {\bf 37}(1982), no. 3, 303-342.



\bibitem[HL]{HL}
Q. Han and F. H. Lin,
Elliptic partial differential equations, second edition.
American Mathematical Society, 2011.

\bibitem [K]{K}
S. Ko\l odziej,
The complex Monge-Amp\`ere equation.
{\it Acta Math.} {\bf 180}(1998), 69-117.

\bibitem [L1]{L1}
N. Le,
Remarks on the Green's function of the linearized Monge-Amp\`ere operator.
{\it Manuscripta Math.} {\bf 149}(2016), no. 1, 45-62.

\bibitem [L2]{L2}
N. Le,
H\"older regularity of the 2D dual semigeostrophic equations via analysis of linearized Monge-Amp\`ere equations.
{\it Comm. Math. Phys.} {\bf 360}(2018), no. 1, 271-305.

\bibitem[LSW]{LSW}
W. Littman, G. Stampacchia and H. Weinberger,
Regular points for elliptic equations with discontinuous coefficients.
{\it Ann. Scuola Norm. Sup. Pisa (\uppercase\expandafter{\romannumeral3})} {\bf 17}(1963), 43-77.

\bibitem [M1]{M1}
D. Maldonado, 
On the $W^{2,1+\epsilon}$-estimates for the Monge-Amp\`ere equation and related real analysis.
{\it Calc. Var. PDE} {\bf 50}(2014), 93-114.

\bibitem [M2]{M2}
D. Maldonado, 
The Monge-Amp\`ere quasi-metric structure admits a Sobolev inequality.
{\it Math. Res. Lett.} {\bf 20}(2014), no. 3, 527-536.

\bibitem[MS]{MS}
M. Murthy and G. Stampacchia,
Boundary value problems for some degenerate-elliptic operators.
{\it  Ann. Mat. Pura Appl.} {\bf 80}(1968), 1-122.

\bibitem[U1]{U90}
Urbas, J., 
On the existence of non classical solutions for two classes of fully nonlinear elliptic equations, 
{\it Indiana Univ. Math. J.} {\bf 39}(1990), 355-382.

\bibitem[U2]{U01}
Urbas, J., 
An interior second derivative bound for solutions of Hessian equations. 
{\it Calc. Var. Part. Diff. Eqns.} {\bf 12}(2001), 417-431.


\bibitem[S]{S}
T. Schmidt,
$W^{2,1+\epsilon}$-estimates for the Monge-Amp\`ere equation.
{\it Adv. Math.} {\bf 240}(2013), 672-689.

\bibitem[T1]{T1}
N. Trudinger,
On the regularity of generalized solutions of linear, non-uniformly elliptic equations.
{\it Arch. Ration. Mech. Anal.} {\bf 42}(1971), 50-62.

\bibitem[T2]{T2}
N. Trudinger,
Linear elliptic operators with measurable coefficients.
{\it Ann. Scuola Norm. Sup. Pisa Cl. Sci. (3)} {\bf 27}(1973), 265-308.

\bibitem [TW]{TW}
G. Tian and X.-J. Wang, 
A class of Sobolev type inequalities.
{\it Meth. Appl. Anal.} {\bf 15}(2008), 257-270.

\bibitem [TW1] {TW1}
N. Trudinger and  X.-J. Wang, 
The Bernstein problem for affine maximal hypersurfaces.
{\it Invent.  Math.} {\bf 140}(2000), no. 2, 399-422.

\bibitem [TW2] {TW2}
N. Trudinger and  X.-J. Wang, 
The affine plateau problem.
{\it J. Amer. Math. Soc.} {\bf 18}(2005), no. 2, 253-289.


\bibitem[W]{W}
X.-J. Wang, 
The $k$-Hessian equation.
{\it Lecture Notes in Mathematics} 
vol. 1977. Springer, Berlin, 2009.

\bibitem[WWZ1]{WWZ1}
J. X. Wang, X.-J. Wang and B. Zhou, 
Moser-Trudinger inequality for the complex Monge-Amp\`ere equation.
{\it Jour. Func. Anal.} {\bf 279}(2020) 108765.


\bibitem[WWZ2]{WWZ2}
J. X. Wang, X.-J. Wang and B. Zhou, 
A Priori Estimate for the Complex Monge–Amp\`ere Equation. 
{\it Peking Math. Jour.}  {\bf 4}(2021), 143-157.



\bibitem[X]{X}
Y. L. Xu,
Interior H\"older estimate for the linearized complex Monge-Amp\`ere equation.
\textit{arXiv:2301.01793}.

\end{thebibliography}
\end{document}